\documentclass{article}

\usepackage{fullpage}
\usepackage{amsmath}
\usepackage{amssymb}
\usepackage{amsfonts}
\usepackage{amsthm}
\usepackage{enumitem}
\usepackage{graphicx}
\usepackage{todonotes}
\usepackage{xcolor}
\usepackage{hyperref}
\usepackage[labelfont=bf]{caption}
\usepackage[capitalise]{cleveref}
\usepackage{comment}
\usepackage{cases}

\def\bE{\mathbb{E}}

\def\dt{\Delta t}

\newcommand{\s}{{\sigma_*}}
\newcommand{\quand}{\quad \text{and} \quad}

\newcommand{\R}{\mathbb{R}}
\newcommand{\N}{\mathbb{N}}

\newcommand{\proj}{\text{proj}}
\newcommand{\p}{\partial}
\newcommand{\bP}{{\bf P}}
\newcommand{\bI}{{\bf I}}
\newcommand{\bQ}{{\bf Q}}

\newcommand{\itc}[1]{#1}

\renewcommand{\Pr}{\operatorname{Pr}}

\def\bsigma{\boldsymbol{\sigma}}
\def\balpha{\boldsymbol{\alpha}}
\def\bbeta{\boldsymbol{\beta}}

\newtheorem{proposition}{Proposition}
\newtheorem{theorem}{Theorem}

\newtheorem{lemma}{Lemma}
\newtheorem{remark}{Remark}

\newtheorem{definition}{Definition}

\title{Modeling Information Flow with a Multi-Stage Queuing Model
\thanks{This research has been supported, in part, by the NSF grant DMS-1620278 and, in part, by the U.S. Department of Energy, Office of Science, Office of Advanced Scientific Computing Research (ASCR), as part of their Applied Mathematics Research Program. Work supported by ASCR was performed at the Oak Ridge National Laboratory, which is managed by UT-Battelle, LLC under Contract No. De-AC05-00OR22725. The United States Government retains and the publisher, by accepting the article for publication, acknowledges that the United States Government retains a non-exclusive, paid-up, irrevocable, world-wide license to publish or reproduce the published form of this manuscript, or allow others to do so, for the United States Government purposes. The Department of Energy will provide public access to these results of federally sponsored research in accordance with the DOE Public Access Plan (http://energy.gov/downloads/doe-public-access-plan). }
}
\author{Mohammad Daneshvar\thanks{Dept. of Mathematics, University of Houston, Houston, TX, 77204, (mdanesh@math.uh.edu)}\quad
Richard C. Barnard\thanks{Dept. of Mathematics, Western Washington University, Bellingham, WA, 98225, (barnarr3$@$wwu.edu)} \quad
Cory Hauck\thanks{Computer Science and Mathematics Division, Oak Ridge National Laboratory, Oak Ridge, TN 37831 USA,
(hauckc$@$ornl.gov)}\quad 
Ilya Timofeyev\thanks{Dept. of Mathematics, University of Houston, Houston, TX 77204, (itimofey@central.uh.edu)}}
\date{}

\begin{document}

\maketitle

\begin{abstract}
    In this paper, we introduce a nonlinear stochastic model to describe the propagation of information inside a computer processor.
    In this model, a computational task is divided into stages, and information can 
    flow from one stage to another.
    The model is formulated as a spatially-extended, continuous-time Markov chain where space represents different stages. 
    This model is equivalent to a spatially-extended version of the M/M/s queue. The main modeling feature is the throttling function which describes the processor slowdown when the amount of information falls below a certain threshold.
    We derive the stationary distribution for this stochastic model and develop a closure for a deterministic ODE system that approximates the evolution of the mean and variance of the stochastic model. 
    We demonstrate the validity of the closure with numerical simulations.
\end{abstract}


\section{Introduction}

Spatially extended Markov chains
have been used to describe various physical systems such as car traffic models (e.g., \cite{soka06,hst2014,sunti14}), pedestrian models (e.g., \cite{ckpt2014,otw2019}), epitaxial crystal growth (e.g., \cite{sce11}), and order formation in bacterial colonies (e.g., \cite{lega13,lewe14,kotabj19}). Typically, these models are formulated as a collection of agents that move on a lattice according to local and global interaction rules modeling various physical phenomena. In this context, many transport models have been introduced and extensively analyzed, such as the Moran model \cite{moran1962} and Asymmetric Exclusion Processes \cite{spitzer1970}. 
Often one of the key questions is a derivation of
the corresponding fluid model using a mean-field approach or various other approximations (e.g., see review papers \cite{bpt12,helbing2001,css00}, and \cite{schchni10}).

In this paper, we construct and analyze a stochastic agent-based model for the flow of information inside a computer processor. We assume that a processor executes a single job that can be divided into $N$ separate approximately identical serial tasks (or stages). For instance, one can think of the time steps in a time-stepping algorithm for a dynamical system or the iterates of an iterative solver.  In this scenario, information enters the processor at a given rate and then propagates from one computational stage to another. 

At the microscopic level, a processor can be modeled by a one-dimensional discrete lattice with $N$ sites, which we refer to as stages, and then the execution of task $k$, for $1\leq k \leq N$, is modeled as the movement of data from stage $k$ to stage $k+1$. Here stage $N+1$ is considered the output stage from which the information is removed from the system. This model is a spatially extended continuous-time Markov chain equivalent to the spatially-extended M/M/s queue (see, e.g., \cite{mcbook13}). Stochasticity in the system is motivated by the fact that the execution time of a task is not deterministic due to variations in the processor due to outside effects or variations in the task itself based on the values of the data being moved.

A more realistic model would include a multi-processor parallel architecture. 
Since a single queue describes each processor, this would lead to 
many such queues coupled with each other using rules that describe how information is exchanged between processors.
A deterministic version of this model was considered in \cite{hauckdata1, hauckdata2}.
However, a stochastic analog of the parallel architecture is much more complex than a single queue, 
and we will consider it in subsequent papers.

To model the slowdown in overall processing speed due to a lack of information, we introduce a throttling function that limits the rate at which the processor operates when not enough information is available. The result is a continuous, piece-wise linear, stochastic transport model with two different operating regimes. The model is a Markov chain that can be recast as a multi-stage (tandem) queue (see e.g. \cite{bby18}), where customers have to go through all stages and each stage represents an M(t)/M(t)/s queue.
Thus, our work is closely related to the queueing theory, including non-stationary applications (see reviews \cite{omega2016a, omega2016b}).
It has been demonstrated in \cite{daneshvar2021} that approximations based on the usual mean-field approach fail to reproduce the average behavior of the stochastic model.  As an alternative, we introduce a statistical moment model for the mean and variance of the Markov chain and use a negative binomial approximation as a closure.  
The result is a system of ordinary differential equations that can be easily simulated. However, it is crucial to demonstrate that the deterministic model is well-posed; that is, one must verify that the model will produce a unique solution and that such a solution will not include unrealistic states--e.g. negative means or negative second moment. \itc{Therefore, the main emphasis of the analysis in Section \ref{sec3.2} is to show that the ODE model provides a single, physically realizable solution for any initial state.
In particular, the main analytical results are given in Theorems \ref{th1}
(existence and uniqueness of solutions)
and \ref{th2}
(invariance of the domain).}
We also demonstrate numerically that the resulting system of ordinary differential equations adequately reproduces the statistical behavior of the stochastic model.

The novel contribution of this paper consists of considering a spatially-extended system (tandem queue) where each stage is an M/M/s queue and deriving a deterministic closure for this model. Following the work of Rothkopf and Oren \cite{rothkopforen1979} for a single M/M/s queue, we develop a closure for the spatially-extended version. In addition, in our model, the number of servers, $s$, can depend on the stage. We also address the well-posedness of the deterministic system and prove the domain invariance, which guarantees physically relevant solutions. The deterministic system derived here can be used to efficiently compute the mean waiting time and the mean response time for the general
multi-stage queue with a varying (with respect to the stage) number of servers, $s$.

The remainder of the paper is organized as follows.  In \Cref{sec2}, we introduce the model, formulate it as a continuous time Markov chain, and then find its stationary distribution.  In \Cref{sec3}, we introduce the closure used to derive an approximate ODE system for the mean and variance at each stage of the Markov chain.  In \Cref{sec4}, we present the natural extension of these results to general stages with a stage-dependent thresholding parameter.  In \Cref{sec5}, numerical results are presented to demonstrate the ability of the ODE system to accurately capture the mean and variance of the stochastic model.  In \Cref{sec6}, we make conclusions and discuss future work.  A short Appendix provides the details of a calculation needed for one of the proofs in \Cref{sec3}.

\section{Model for the Flow of Information}
\label{sec2}

We model the information in a processor as a homogeneous continuous-time Markov chain with $N \geq 1$ stages and let $\sigma_{k}(t) \in \N = \{0,1,2,\ldots\}$ be the units of data at stage $k \in \{1,\dots, N\}$ of the processor at time $t$. We assume that data is neither created nor destroyed, only moved in and out of the processor or between the stages. The eventual goal is to obtain a macroscopic description to estimate the mean 
$\mathbb{E}[ \sigma_{k}(t) ]$.

Let $\Sigma ={\N}^{N}$ be the state space of the Markov chain, and denote the configuration of the chain at time $t$ by
$\bsigma(t)=( \sigma_{1}(t), \sigma_{2}(t),\ldots,\sigma_{N}(t))$.
We introduce transition probabilities to describe information flow from one stage to another.  
We assume that one unit of data at any stage is processed and then moves to the 
next stage independently from other units of data in the processor, whether at the same stage or at a different stage. For $i \in \{0, \dots ,N\}$, let $T_i \colon \N^N \to \N^N$ be the operator that models the propagation of one unit of data from stage $i$ to stage $i + 1$, where stage $0$ is the input and stage $N+1$ is the output.  More specifically, for each $j \in \{1, \dots ,N\}$, let
\begin{align}
\label{eq:T}
( T_i (\balpha) )_j  &
= \begin{cases}
  \alpha_j - 1, &\quad j = i,  \\
  \alpha_j + 1, &\quad j = i+1 \\
  \alpha_j, &\quad \text{otherwise} .
 \end{cases} 
\end{align}   
Then the stochastic dynamics of the data flow model are determined by the following probabilities:
\begin{subequations}
\label{eq:jump_prob}
\begin{numcases}{\Pr \left[\bsigma(t+\dt)= T_k \balpha \, | \,
\bsigma(t) =\balpha \right] = }
\label{eq:1} c_{0}(t) \dt + \mathcal{O}(\dt^2), & $\quad k=0$,
\\
\label{eq:2}f(\alpha_{k})\dt + \mathcal{O}(\dt^2), & $\quad k=1,\ldots,N-1$, 
\\
\label{eq:2a}f(\alpha_N) \dt + \mathcal{O}(\dt^2), & $\quad k=N$. 
\end{numcases}
\end{subequations}
Because the moves are independent, the probability of multiple moves in the time interval $\dt$ is  $\mathcal{O}(\dt^{2})$. 
The system in \eqref{eq:jump_prob} is a standard continuous-time Markov chain corresponding to uni-directional transport.
Equations \eqref{eq:1} and \eqref{eq:2a} describe the input and output of information
to and from the processor, respectively, while Equation \eqref{eq:2} describes the flow of information between stages.
%

The input function $c_0(t)$ can be rather general since it only defines the input of information and does not affect how it is processed and moved between different stages.  While it is not necessary,  we will assume that $c_0$ does not depend on time for the analytical results in Section \ref{sec222}.
The function $f(t)$, on the other hand, is essential since it defines the transport of information and, thus, the nonlinear nature of the model. We will discuss this in the next section.

\subsection{Processor Throttling}
\label{sec2.1}

If there is sufficient information available at a given processor stage, then it will be processed at a constant rate.
However, a slowdown occurs if the amount of data available drops below a certain threshold $\sigma_{*} >0$.  In such cases, the computational unit cannot maintain its maximum rate, and the information processing speed is reduced.  This phenomenon is referred to as processor throttling. We model it using a throttling function
\begin{equation}
\label{v}
v(x)=\begin{cases} 
      0, & \quad x \leq 0, \\
      \frac{x}{\sigma_{*}}, & \quad 0\leq x \leq \sigma_{*},\\
      1, & \quad x \geq \sigma_{*},
   \end{cases}
\end{equation}
where the throttling parameter $\sigma_{*} \geq 1$ is an integer value. 
\itc{Figure \ref{fig:2} depicts an example 
of the throttling function $v$.}
Then the rate of information transfer between stages becomes 
\begin{equation}
\label{f}
    f(\alpha_k) = c\,v(\alpha_k), \quad  k=1,\ldots,N-1,
\end{equation}
where $c$ is a constant that defines 
the maximum rate at which information can be processed.  The function $f$ is a monotonically decreasing function of $\sigma_*$; thus the rate decreases when $\sigma_*$ increases.  This fact will be reflected in the behavior of the numerical results in Section \eqref{sec5}. 
\begin{figure}[h]
\centering
\includegraphics[scale=0.6]{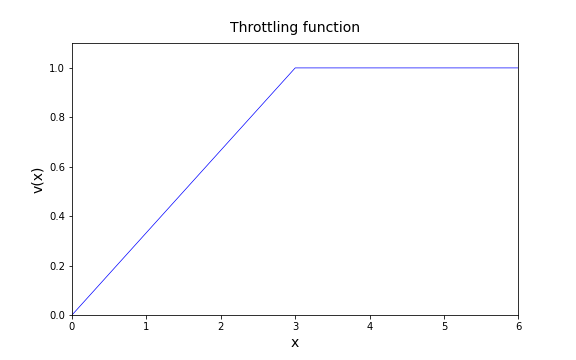}
\caption{Throttling function $v$, defined in \eqref{v}, with $\sigma_*=3$. A slowdown in the processing rate occurs when $x < \sigma_*$.}
\label{fig:2}
\end{figure}


\subsection{Stochastic Model as a Continuous-Time Markov Chain}
\label{sec2.2}

The stochastic model defined in \ref{sec2} with the throttling function
\eqref{f} constitutes a continuous-time Markov chain $\bsigma(t)$ with values on $\N^N$.   For completeness, we briefly summarize basic definitions and results from the theory of continuous-time Markov chains that will be used below.

\subsubsection{Theory of continuous-time Markov chains}
Here we follow the discussion in \cite{grimstir92,norris97}. 
While the results are formulated for scalar-valued Markov chains,
they can easily be mapped to vector-valued chains like those introduced in Section
\ref{sec2} by reindexing. First, we introduce the notion of the stochastic semigroup.
\begin{definition} The transition probability of a Markov chain $X_t$ is given by
\begin{equation}
    p_{ij}(s,t) = \Pr(X_t = j|X_s = i), \quad \text{for} \quad s \le t.
\end{equation}
The chain is called homogeneous if
$p_{ij}(s,t) = p_{ij}(0,t-s)$
for all $s \le t$.
\end{definition}
It is easy to see that the Markov chain $\bsigma(t)$ is homogeneous if the input rate $c_0$ does not depend on time.
For homogeneous chains we denote $p_{ij}(t) := p_{ij}(0,t)  $ and write $\bP_t$   for the (possibly infinite) matrix with entries $p_{ij}(t)$. 
\begin{theorem}
\cite[Theorem 6.9.3]{grimstir92}.
    The family $\{\bP_t: t \ge 0\}$ is a stochastic semigroup; i.e., it satisfies 
    \begin{enumerate}[label=(\roman*)]
        \item  $\bP_0 = \bI$ (identity);
    \item  $\bP_t$ is stochastic, i.e., $\bP_t$ has non-negative entries and each row sums to 1;
    \item  $\bP_t$ satisfies the Chapman-Kolmogorov equation
    $\bP_{s+t} = \bP_s \bP_t$ for all $s,t \ge 0$.
    \end{enumerate}
   
\end{theorem}
The next few definitions and theorems establish the basic properties of continuous-time Markov chains.
\begin{definition}
The semigroup $\{\bP_t\}$ is called standard if 
$\bP_t \to \bI$ as $t\to 0$. It is called uniform if 
$
\bP_t \to \bI$ uniformly (with respect to the elements) as $t\to 0.  
$
\end{definition}
\begin{remark}
If $\{\bP_t\}$ is uniform, then it is also standard.
\end{remark}
\begin{definition}
The generator $\bQ$ of the Markov chain is given by
\begin{equation}
\label{Qdef}
   \bQ = \lim_{h\to 0} \frac{\bP_h - \bI}{h}. 
\end{equation}
The elements $\bQ$ are denoted by $q_{ij}$.
\end{definition}
Transition probabilities can be expressed as $\bP_t = e^{t \bQ}$ and
$\bP_h = \bI + h \bQ + o(h)$ for small $h$. Therefore, 
it is easy to show that the non-diagonal entries of the generator correspond to jump rates, i.e., $p_{ij}(h) \approx q_{ij} h$ for $i \ne j$, and that each row of the generator sums to 0. Thus, the diagonal entries of the generator are negative.
\begin{theorem} 
\label{thm2}\cite[Theorem 6.10.5]{grimstir92}.
$\{\bP_t\}$ is uniform if and only if the diagonal entries of the generator are uniformly bounded, i.e.,
\begin{equation}
\sup_i |q_{ii}| < \infty.
\end{equation}
\end{theorem}

Next, we introduce the classification of states and the notion of an irreducible Markov chain.
\begin{definition} 
\label{def: irreducible} Let $i$ and $j$ be two possible states of a Markov Chain $X_t$.  We say that
$i$ leads to $j$ (written $i \to j$) if $\Pr(X_t = i | X_0 = j) > 0$ for some $t \ge 0$.  We say that $i$ communicates with $j$ (written $i \leftrightarrow j$) if both $i \to j$ and $i \to j$.  A Markov chain that consists of a single communicating class is called irreducible. 
\end{definition}
%
%


Finally, we state the notion of the stationary distribution and the uniqueness theorem. 
\begin{definition} The vector $\pi$ is called a stationary distribution if $\pi_j \ge 0$, $\sum_j \pi_j = 1$, and $\pi = \pi \bP_t$ for all $t\ge 0$.
\end{definition}
\begin{theorem}
\cite[Section 6.9]{grimstir92}
$\pi = \pi \bP_t$ if and only if $\pi \bQ = 0$.  
\end{theorem}
\begin{theorem} 
\label{thm4}
\cite[Theorem 6.9.21]{grimstir92}.
Let $X_t$ be irreducible with a standard semigroup 
$\{\bP_t\}$ of transition probabilities. If there exists a stationary distribution $\pi$, then it is unique and $p_{ij}(t) \to \pi_j$ as $t\to\infty$ for all $i$ and $j$.
\end{theorem}

\subsubsection{Analysis of the Markov chain}
\label{sec222}
In this section, we assume that the input rate $c_0$ (see \eqref{eq:1}) (i) does not depend on time and (ii) is bounded above by the maximum processor rate $c$ introduced in Section \ref{sec2.1}.  These assumptions are sufficient to ensure the existence of a unique stationary distribution and to derive an explicit expression for that distribution.  
However, the modeling and simulations results in Sections \ref{sec3.2} and \ref{sec5} allow for the more general case that $c_0$ is time-dependent.

The stochastic process $\sigma_1(t)$
can be analyzed separately from the rest of the stages since it depends only on the input function $c_0$ and is therefore decoupled from the remaining stages of $\bsigma(t)$.
Therefore, to develop an intuition about the behavior of the Markov chain $\bsigma(t)$
we briefly consider $\sigma_1(t)$ separately.
The infinitesimal generator for $\sigma_1(t)$
$\mathbf{Q}$ is given by
\begin{equation}
\begin{split}  
\mathbf{Q} &= \begin{bmatrix}
-c_0 & c_0& 0 &0& \dots\\
c\,v(1)&-c_0-c\,v(1)&c_0&0&\dots\\
0& c\,v(2)&-c_0- c\,v(2)&c_0&\dots\\
0&0& c\,v(3)&-c_0- c\,v(3)&\dots \\
\vdots&\vdots&\vdots&\vdots&\vdots
\end{bmatrix}.
\end{split}
\label{Q}
\end{equation}
The stochastic process $\sigma_1(t)$ can be reformulated as a birth-death process with birth rates $\lambda_k \equiv q_{k,k+1} = c_0$ and death rates $\mu_k \equiv q_{k,k-1} = cv(k)$.
From the formula for $v$ in \eqref{v} and under the assumption $c_0<c$, it follows that
\begin{equation}
\label{eq:sum_ratio}
\sum\limits_{n=0}^{\infty} \frac{\lambda_0\lambda_1 \cdots \lambda_{n-1}}{\mu_1\mu_2 \cdots \mu_{n}} 
= 1 + \sum\limits_{n=1}^{\sigma_*-1} 
\frac{c_0^n }{c^n } \frac{\sigma_*^n}{n!} ++
\sum\limits_{n=\sigma_*}^{\infty} 
\frac{c_0^n}{c^n}
< \infty,
\end{equation}
%
where the term above corresponding to $n=0$ is understood as 1.
The condition \eqref{eq:sum_ratio} is sufficient (see Section 6.11 of \cite{grimstir92})
for $\sigma_1(t)$ to have unique stationary distribution $\pi$.  It is straightforward to verify that $\pi$ given by
\begin{subequations}
\label{eq:pi}
\begin{align}
\pi_{0} = \lim_{t \to \infty} \Pr[\sigma_1(t) = 0] = & \left( {1+\sum_{j=1}^{\infty}\frac{ c_0^{j}}{c^{j} \prod_{k=1}^{j}v(k) }} \right)^{-1},
\label{pi0} \\
\pi_{j} = \lim_{t \to \infty} \Pr[\sigma_1(t) = j] = & \frac{ c_0^{j}}{c^{j} \prod_{k=1}^{j}v(k) }
\left( {1+\sum_{j=1}^{\infty}\frac{ c_0^{j}}{c^{j} \prod_{k=1}^{j}v(k) }}\right)^{-1}, \qquad j\geq 1,
\label{pik}
\end{align}
\end{subequations}
satisfies $\pi \mathbf{Q} = 0$.

Next, we consider the full Markov chain $\bsigma(t)$ and discuss its basic properties.

\begin{lemma}
\label{lem:irreducible}
The continuous time Markov chain $\bsigma(t)$ is irreducible. 
\end{lemma}
\begin{proof}
Define the function $G^{(i)}_{\bbeta} \colon \N^N \to \N^N$ by
\begin{equation}
    (G^{(i)}_{\bbeta} (\balpha))_j  =
    \begin{cases}
        \beta_j , & \quad j = i, \\ 
        \alpha_j , & \quad j \ne i.
        \end{cases}
\end{equation}
According to Definition \ref{def: irreducible}, it is sufficient to show that for any $\balpha \in \N^N$ and $\bbeta \in \N^N$
\begin{equation}
\label{eq:leads}
\Pr[\bsigma(t + \Delta t) = G^{(i)}_{\bbeta}(\balpha)  \, | \, 
    \bsigma(t) = \balpha]  > 0
\end{equation}
for $\dt$ sufficiently small.  Recall the  definition of $T_i$ given in \eqref{eq:T}, and let $l_i = \alpha_i -  \beta_i$.  
\itc{
Then
\begin{subequations}
\begin{numcases}{G^{(i)}_{\bbeta} (\balpha)=}
\label{eq:l>0}
[T_N]^{l_i} 
       \circ [T_{N-1}]^{l_i}
       \dots
       \circ [T_{i+1}]^{l_i}
       \circ [T_{i}]^{l_i}
       (\balpha),
& if $l_i > 0$, \\
\label{eq:l<0}
    [T_{i-1}]^{|l_i|} \circ  [T_{i-2}]^{|l_i|} \circ \dots [T_{1}]^{|l_i|}  \circ [T_{0}]^{|l_i|} (\balpha),
& if $l_i < 0$ \\
\label{eq:l=0}
\balpha, & if $l_i = 0$.
\end{numcases}
\end{subequations}
Each of the transitions in \eqref{eq:l>0} is of the type in \eqref{eq:2} or \eqref{eq:2a}, while each of the transitions in \eqref{eq:l<0} is of the type in \eqref{eq:1} or \eqref{eq:2}.  In either case, the probability of each such transition is positive for $\Delta t$ sufficiently small. In \eqref{eq:l=0}, there is no transition, and this happens with probability $1-\mathcal{O}(\Delta t)$.  Thus \eqref{eq:leads} holds for $\Delta t$ sufficiently small.  
}
Finally, it is clear that
\begin{equation}
    \bbeta = G^{(1)}_{\bbeta} \circ G^{(2)}_{\bbeta} \circ
    \dots \circ G^{(N-1)}_{\bbeta} \circ G^{(N)}_{\bbeta} (\balpha),
\end{equation}
and therefore
\begin{equation}
    \Pr[\bsigma(t + \Delta t) = \bbeta  \, | \, 
    \bsigma(t) = \balpha]
    > 0
\end{equation}
for $\Delta t$ sufficiently small.
\end{proof}

\begin{lemma}\label{lem:semigroup}
The transition probabilities associated with $\bsigma(t)$ form a uniform semigroup. 
\begin{proof} 
According to Theorem \ref{thm2}, it is sufficient to show that diagonal entries of the generator are uniformly bounded.
Let $\mathbf{Q}$ be the generator for $\bsigma(t)$ with elements $q_{\balpha,\bbeta}$ indexed by ${\balpha,\bbeta} \in \N^N$.  According to \eqref{eq:jump_prob} and the assumption that $c_0 \leq c$, for any state $\balpha \in \N^N$, $t \geq 0$, and $\dt>0$ sufficiently small,  
\begin{subequations}
\begin{align}
        \label{eq:Ti_trans}
        \Pr[\bsigma(t + \dt) = T_i(\balpha) \,|\, \bsigma(t ) = \balpha] & \leq c \dt + \mathcal{O}(\dt^2), \qquad i \in \{0 , \dots, N\} ,\\
    \Pr[\bsigma(t + \dt) \not \in \{\balpha,T_1 \balpha, \dots, T_N \balpha\} \,|\, \bsigma(t ) = \balpha] &= \mathcal{O}(\dt^2).
\end{align}
\end{subequations}
Thus for each $\balpha$, the only transitions which lead to off-diagonal non-zero entries of $Q$ are given in \eqref{eq:Ti_trans}; there are only $N+1$ such transitions. Furthermore, by definition, $\sum_{\bbeta} q_{\balpha,\bbeta}=0$.  Hence
\begin{equation}
   |q_{\balpha,\balpha} |
   =  |\sum_{\bbeta \ne \balpha} q_{\balpha,\bbeta} | 
   \leq \sum_{\bbeta \ne \balpha} |q_{\balpha,\bbeta}| 
   \leq (N+1)c.
\end{equation}
\end{proof} 
\end{lemma}
In the following theorem, we extend the result regarding the invariant distribution for the Markov chain $\bsigma(t)$.
\begin{theorem}
\label{prop1}
The Markov chain $\bsigma(t)$ has a unique
stationary distribution given by 
\begin{equation}
\pi_{\balpha} = C \,\frac{ c_0^{\alpha_1}}{c^{\alpha_1} \prod_{k=1}^{\alpha_1}v(k) } \times \frac{ c_0^{\alpha_2}}{c^{\alpha_2}  \prod_{k=1}^{\alpha_2}v(k) } \times \cdots \times \frac{ c_0^{\alpha_N}}{c^{\alpha_N} \prod_{k=1}^{\alpha_N}v(k) } 
\label{eq:25}
\end{equation}
where $C$ is a normalization constant given by
\begin{align}
C&=  
\left[ 
\sum_{\alpha_1=0}^\infty \dots \sum_{\alpha_N=0}^\infty 
\left( \frac{ c_0^{\alpha_1}}{c^{\alpha_1} \prod_{k=1}^{\alpha_1}v(k) } \times \frac{ c_0^{\alpha_2}}{c^{\alpha_2} \prod_{k=1}^{\alpha_2}v(k) } \times \cdots \times \frac{ c_0^{\alpha_N}}{c^{\alpha_N}  \prod_{k=1}^{\alpha_N}v(k) }
\right) 
\right]^{-1},
\end{align}
and $\prod_{k=1}^{\alpha_j}v(k) \equiv 1$ if $\alpha_j=0$.\\
\end{theorem}
\begin{proof} 
Let $\bP_t$ be the transition matrix for $\bsigma(t)$ with elements $p_{\balpha, \bbeta}(t)
=\Pr[\bsigma(t)=\balpha| \bsigma(0)=\bbeta] $, and for $i \in \{0,\dots,N\}$, let $S_i = T_i^{-1}$, where $T_i$ is defined in \eqref{eq:T}.
Under the assumption that $\Pr[\sigma_i(t) < 0]$ is identically zero for all $i \in \{0,\dots,N\}$, the Kolmogorov forward equation for $\bP_t$ is
\begin{align}
\begin{split}
\partial_{t} p_{\balpha, \bbeta}(t)
&=c_0 \left[p_{(S_0\balpha),\bbeta}(t) -  p_{\balpha,\bbeta}(t) \right]
+ c \sum_{i=1}^{N} \left[ v(\alpha_{i}+1) p_{(S_i\balpha), \bbeta} (t)
- v(\alpha_{i})  p_{\balpha,\bbeta}(t) \right].
\end{split}
\label{eq:26}
\end{align}
%
To find the stationary solution, we use equation \eqref{eq:26} with the 
left-hand side set to zero. It can be shown by direct substitution that \eqref{eq:25} solves
the equation for the stationary distribution.

 We next prove uniqueness.  According to Lemma \ref {lem:irreducible}, $\bsigma(t)$ is irreducible.  According to Lemma \ref{lem:semigroup}, the associated transition probabilities form a uniform, and hence standard, semi-group.  Therefore $\bsigma(t)$ satisfies the conditions of Theorem \ref{thm4}, from which the uniqueness of the stationary distribution follows.  
 
\end{proof}
Because the distribution \eqref{eq:25} has a product form, it follows that when $\bsigma(t)$ is stationary, all stages are independently identically distributed. Hence, the moments of $\sigma_k(t)$ with respect to the steady state distribution are 
\begin{equation}
    \lim_{t \to \infty} \bE[(\sigma_k(t))^{n}] = 
\left(1 + \sum\limits_{\alpha_k=1}^\infty
\frac{ c_0^{\alpha_k}}{c^{\alpha_k}  \prod_{i=1}^{\alpha_k}v(i) }\right)^{-1} \sum_{\alpha_k=1}^{\infty}(\alpha_k)^{n}\,\frac{ c_0^{\alpha_k}}{c^{\alpha_k} \prod_{i=1}^{\alpha_k}v(i)}.
\label{eq:28}
\end{equation}

\subsection{Stochastic Model as a Multi-Stage Queue}
\label{queue}
The microscopic model introduced in Section \ref{sec2.1}
can be interpreted as a multi-stage M/M/s queue, and therefore it can be analyzed in the context of queuing theory \cite{adan2002queueing,bhat2015}. If we identify the computational stages in our model with queues, then
our model is equivalent to $N$ consecutive queues. The job enters the first queue, $\sigma_1$, and after it is served in $\sigma_1$, it enters the next queue, $\sigma_2$, and so on.
In particular, the first stage, $\sigma_1(t)$ is a classical M/M/s queue where
jobs arrive with the arrival rate $c_{0}(t)$ and the threshold parameter $\sigma_*$ can be interpreted as the number of servers in the queue. The utilization of the server at time $t$ is $c_{0}(t)/c$ and the average service rate is $c/\sigma_*$.


In the M/M/s interpretation, the queue has the same number of servers at each stage because the threshold parameter $\sigma_*$ is the same. However, for the numerical results in Section \ref{sec4}, we allow for a more general case in which each queue can have a different number of servers; that is, we allow $\sigma_*$ to be a function of the stage $k$.

\section{Deterministic Closure}
\label{sec3}
The goal of this section is to construct a closure to approximate the dynamics of the expected value $\bE_k(t):= \bE[\sigma_k(t)]$ for each $k \in  \{1,\dots, N\}$.
To derive equations for $\bE_k(t)$,
we use the properties of the generator and the Dynkin formula.
In particular, for any function $\phi: \mathbb{N}^N \to \mathbb{N}$, 
\begin{equation}
\frac{d}{dt} \bE [\phi(\bsigma(t))] 
= \bE [\bQ \phi(\bsigma(t))]
= \sum_{\bbeta} \Pr[\bsigma(t) = \bbeta] \sum_{\balpha} \bQ_{\bbeta,\balpha} \phi(\balpha), 
\end{equation}
and, rather than using \eqref{Qdef} directly, we compute the action of the generator $\bQ$ using the formula
\begin{equation}
[\bQ \phi(\bsigma(t))]_{\bbeta} = \lim_{h \to 0} \frac{\bE [\phi(\bsigma(h))|\bsigma(0) = \bbeta] - \phi(\bbeta)}{h} .
\end{equation}
Setting $\phi(\bsigma(t)) = \sigma_k(t)$, $k=1,\ldots,N$ yields following system of equations for $\bE_k(t)$:
    
%
\begin{subequations}
\label{eq:29}
\begin{align}
\frac{d}{dt} \bE_1(t) & = c_0(t)- c\,\bE[v(\sigma_{1} (t))], \\
\frac{d}{dt} \bE_k(t) & = c\,\bE[v(\sigma_{k-1} (t)) ]- c\,\bE[v(\sigma_{k} (t))], \qquad k=2\dots,N.
\end{align}
\end{subequations}
The equations above are exact but not closed since $\bE[v(\sigma_{k} (t))]$ is an expectation of a 
nonlinear function. 

A naive mean-field assumption holds under the ad-hoc assumption that the expectation commutes with the 
nonlinear function $v$, i.e.
$\bE[v(\sigma_k(t))] = v(\bE[\sigma_k(t)])$. With this assumption, \eqref{eq:29} 
becomes a closed systems of equations:
\begin{subequations}
\label{eq:naive}
\begin{align}
\frac{d}{dt} \bE_1(t) &= c_0(t)- c\,v(\bE[\sigma_{1} (t)]), \\
\frac{d}{dt} \bE_k(t) &= c\,v(\bE[\sigma_{k-1} (t)])- c\,v(\bE[\sigma_{k} (t)]), \qquad k=2\dots,N.
\end{align}
\end{subequations}
%
While in some parameter regimes
(e.g. $\sigma > c_0$)
this naive closure produces adequate results \cite{daneshvar2021}, the discrepancies between the deterministic system \eqref{eq:naive}
and the Monte-Carlo simulations of the stochastic system are significant when $\sigma < c_0$. An example of this behavior is provided by Figure \ref{fig:naive}, where results from the simulation of the 
stochastic system $\bsigma(t)$
and the deterministic system \eqref{eq:naive} are given. 
\itc{Numerical results for both, the stochastic and deterministic models resemble a nonlinear wave phenomenon with the leading front corresponding to information propagating through consecutive stages.}
The difference in wave propagation between the two models indicates that the naive mean-field closure \eqref{eq:naive} is not generally
applicable to the stochastic model $\bsigma(t)$. 
We perform comparison between stochastic simulations and our improved closure model in the same regime $(c,\sigma_*)=(10,3)$
in 
Section \ref{num1} and Figure \ref{fig:p1} in particular.
\begin{figure}[h]
\centering
\includegraphics[scale=0.8]{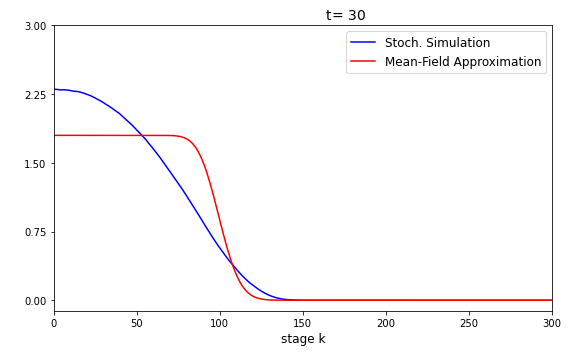}
\caption{Numerical comparison of the mean $\bE[\sigma_k(t)]$ in deterministic system \eqref{eq:naive} and Monte-Carlo simulations of the stochastic system $\bsigma(t)$
for $c=10$, $c_0=6$, $\sigma_* = 3$, $t=30$, zero initial condition. Stochastic results are averaged over 10,000 Monte-Carlo trajectories. C.f. with results in Section \ref{num1} and Figure \ref{fig:p1} in particular,
where numerical results for our improved closure model are discussed.}
\label{fig:naive}
\end{figure}

We proceed instead with a more sophisticated closure to approximate the dynamics on the right-hand side of equations \eqref{eq:29}. First, the expected value of $\sigma_k(t)$ can be rewritten as follows:
\begin{equation}
\begin{split}
\bE[v(\sigma_k(t))]
&=\sum\limits_{i=0}^\infty v(i)\Pr[\sigma_k(t)=i] 
=\frac{1}{\sigma_*}\sum\limits_{i=0}^{\sigma_*-1} i\, \Pr[\sigma_k(t)=i] + \sum\limits_{\sigma_*}^\infty \Pr[\sigma_k(t)=i]\\
     &=\frac{1}{\sigma_*} \sum\limits_{i=0}^{\sigma_*-1} i\, \Pr[\sigma_k(t)=i)]+ 
     \left(1-\sum\limits_{i=0}^{\sigma_*-1}\Pr[\sigma_k(t)=i]\right)
     =1-\frac{1}{\sigma_*} \sum\limits_{i=0}^{\sigma_*-1}(\sigma_*-i)\Pr[\sigma_k(t)=i].
\label{eq:31}
\end{split}
\end{equation}
Thus, \eqref{eq:29} can be rewritten as
\begin{subequations}
    \label{eq:32}
    \begin{align}
    \frac{d}{dt} \bE_1(t) & = c_0(t) - c + \frac{c}{\sigma_*}\, \sum\limits_{i=0}^{\sigma_*-1}(\sigma_*-i)\Pr[\sigma_1(t)=i], \\
    \frac{d}{dt} \bE_k(t)  &= \frac{c}{\sigma_*}\, \sum\limits_{i=0}^{\sigma_*-1}(\sigma_*-i)\Pr[\sigma_k(t)=i] -
    \frac{c}{\sigma_*}\, \sum\limits_{i=0}^{\sigma_*-1}(\sigma_*-i)\Pr[\sigma_{k-1}(t)=i], \qquad k=2,\dots,N.
\end{align}
\end{subequations}
An obvious advantage of the equations in \eqref{eq:32} is that they are written in terms of the first $\sigma_*$
terms of the probability distribution and, therefore, the summations are finite. Hence, to find a deterministic approximation for the right-hand side of \eqref{eq:29}, 
it is sufficient to adequately approximate
\begin{equation}
    \label{eq:34}
    F_{k}(t) = \sum\limits_{i=0}^{\sigma_*-1}(\sigma_*-i)\Pr[\sigma_{k}(t)=i].
\end{equation}
We seek a time-dependent approximation of $\Pr[\sigma_k(t)=i]$
which depends on the first two moments of $\sigma_k(t)$.

\subsection{Negative Binomial Approximation}

The negative binomial distribution \cite{hogg1977probability} is a discrete probability distribution that gives the probability of $n$ failures in a sequence of Bernoulli trials with success probability $p$ until the $r$-th success is achieved.  It is given by the probability mass function
\begin{equation}
    \label{eq:nb}
     B_n(r,p) = \frac{\Gamma(n+r)}{n!\Gamma(r)}\,p^{r}\,(1-p)^{n}, \quad\quad n=0,1,2,...
\end{equation}
where $\Gamma$ is the usual gamma function, $p \in [0,1]$, $n \in \N$, and the definition has been extended from $r \in \N_+ = \{1,2,3,\dots\}$ to all $r>0$. 
The mean $\rho$ and variance $\eta$ of a random variable with mass function $B_n$ are given by
\begin{equation}
    \rho = \frac{r(1-p)}{p} 
    \quand
    \eta =\frac{r(1-p)}{p^2}, 
\end{equation}
respectively.  These relations can be inverted for the parameters $r$ and $p$:
\begin{equation}
\label{eq:hat-p-hat-r}
    p = \hat{p}(\rho,\eta)= \frac{\rho}{\eta}
    \quand
    r = \hat{r}(\rho,\eta) = \frac{\rho^2}{\eta-\rho}.
\end{equation}
Moreover, the function
\begin{equation}
\label{eq:P}
    P_n(\rho,\eta) = B_n(\hat{r}\left(\rho,\eta),\hat{p}(\rho,\eta) \right),
\end{equation}
where $\hat{r}$ and $\hat{p}$ are defined in \eqref{eq:hat-p-hat-r}, is well-defined for all $\rho$ and $\eta$ such that  $0 < \rho < \eta$.  It can be also extended to the rays $\eta = \rho \geq 0$ and $\eta \geq \rho =0 $ as follows:
\begin{itemize}
    \item{\bf Case 1: $\eta =\rho > 0$.}  Let $(\rho,\eta) \in \mathbb{R}^2$ such that $\eta > \rho >0 $ and $(\rho,\eta) \to (\rho^*,\eta^*)$, where $\eta^* =\rho^* >0$.  Then $\hat{r}(\rho,\eta) \to \infty$.  Since $\hat{p}(\rho,\eta) = (1+ \frac{\rho}{\hat r(\rho,\eta)})^{-1}$, it follows that
    \begin{equation}
    \label{eq:nb-alt}
     P_n(\rho,\eta) = B_n(\hat{r}(\rho,\eta),\hat{p}(\rho,\eta)) = \frac{\Gamma(n+\hat r(\rho,\eta))}{\Gamma(\hat r(\rho,\eta)) (\rho+\hat r(\rho,\eta))^n} \frac{\rho^n}{n!} \left(1+ \frac{\rho}{\hat r(\rho,\eta)}\right)^{-\hat r(\rho,\eta)} 
\end{equation}
Hence \cite{piegorsch1990maximum}
\begin{equation}
\label{eq:ray1}
    \lim_{(\rho,\eta) \to (\rho^*, \eta^*) } P_n(\rho,\eta) 
    = \lim_{(\rho,r) \to (\rho^*, \infty) } \frac{\Gamma(n+r)}{\Gamma(r) (\rho+r)^n} \frac{\rho^n}{n!} \left(1+ \frac{\rho}{r}\right)^{-r} 
    = \frac{\rho_*^n}{n!} e^{-\rho_*}.
\end{equation}
 \item{\bf  Case 2: $\eta \geq \rho = 0$.} Let $(\rho,\eta) \in \mathbb{R}^2$ such that $\eta > \rho >0 $ and $(\rho,\eta) \to (\rho^*,\eta^*)$, where $\eta^* \geq \rho^* = 0$.  From the definition of the gamma function, if follows that 
 \begin{equation}
\max_{ \{r >0, \rho \in (0,1)\} } \frac{\Gamma(n+r)}{\Gamma(r)(\rho+r)^n} \left(1+ \frac{\rho}{r}\right)^{-r} \leq C_n
 \end{equation}
where $C_n$ is a finite constant that depends only on $n$.  Hence from \eqref{eq:nb-alt}, it follows that 
\begin{equation}
\label{eq:ray2}
     \lim_{\rho,\eta \to (0,\eta^*)} P_n(\rho,\eta) 
     = \begin{cases}
         1, &\quad  n=0, \\
         0, & \quad n \geq 1.
     \end{cases}
\end{equation}
%
%

\end{itemize}



When $r=1$, the negative binomial distribution reduces a geometric distribution \cite{hogg1977probability}, which is equivalent to the steady-state distribution in \eqref{eq:pi} when $\sigma_*=1$.%
\footnote{This fact can be shown by recalling that $v \equiv 1$ when $\sigma_*=1$ (corresponding to an M/M/1 queue) and setting the success probability for the geometric distribution to be $p = 1 - c_0/c$.} Motivated by this fact, the authors in 
\cite{rothkopforen1977,rothkopforen1979} use the negative binomial distribution to approximate the transient probability of an M/M/1 queue, and then extend the application to queues with more than one server. 

In the current model, the first stage is an M/M/s queue, and all stages have the same steady-state distribution.  We, therefore, take a similar approach to \cite{rothkopforen1977,rothkopforen1979} and use the negative binomial distribution to approximate the probability mass function for $\sigma_{k}(t)$ at each stage $k$:
 \begin{equation}
    \label{eq:46}
    \Pr[\sigma_{k}(t)=n] \approx P_n( \rho_k(t), \eta_k(t))
\end{equation}
where $\rho_k(t)$ and $\eta_k(t)$ are approximations for the mean $\bE_k(t)$ and variance 
\begin{equation}
    \mathbb{V}_k(t) = \bE \left[|\sigma_k(t) - \bE_k(t)|^2 \right],
\end{equation}
respectively, of $\sigma_{k}(t)$.  The exact equations for $\bE_k(t)$ are given in \eqref{eq:29}, while the exact equations for $\mathbb{V}_k$ are
\begin{subequations}
\label{eq:var}
\begin{align}
    \frac{d}{dt} \mathbb{V}_{1}(t) &= c_0(t) + c\, \bE[v(\sigma_{1}(t)] + 2\,c_0\,\rho_{1}(t) - 2\, c\, \bE[\sigma_{1}(t)v(\sigma_{1}(t))], \\
    \frac{d}{dt} \mathbb{V}_k(t) &= c\, \bE[v(\sigma_{k-1} (t))] + c\, \bE[v(\sigma_{k}(t))] + 2\, c\, \bE[\sigma_{k-1}(t)v(\sigma_{k-1}(t))]- 2\, c\, \bE[\sigma_{k}(t)v(\sigma_{k}(t))], \quad k=2,\dots,N.
\end{align}
\end{subequations}

Like the equations for $\bE_k$, the equations for $\mathbb{V}_k$ in \eqref{eq:var} are not closed. To derive a closed system for $\rho_k \simeq \bE_k$ and $\eta_k \simeq \mathbb{V}_k$, respectively, we invoke the negative binomial approximation in \eqref{eq:46}.   Then the closed system is
\begin{subequations}
\label{eq:rho1-eta1-closure}
\begin{align}
    &\frac{d}{dt} \rho_1(t) = c_0(t)- c + \frac{c}{\sigma_*}\, \sum\limits_{i=0}^{\sigma_*-1}(\sigma_*-i)
        P_i(\rho_1(t),\eta_1(t))\\
    &\frac{d}{dt} \eta_{1}(t) = c_0(t) + c - \frac{c}{\sigma_*}\, \sum\limits_{i=0}^{\sigma_{*}-1}\, (2\rho_{1}+1-2i)(\sigma_*-i)P_i(\rho_1(t),\eta_1(t)), 
\end{align}
\end{subequations}
and for $k=2,\dots,N$,
\begin{subequations}
\label{eq:rhok-etak-closure}
\begin{align}
    &\frac{d}{dt} \rho_k(t) 
    = \frac{c}{\sigma_*}\, \sum\limits_{i=0}^{\sigma_*-1}(\sigma_*-i)P_i(\rho_k(t),\eta_k(t)) 
    -\frac{c}{\sigma_*}\, \sum\limits_{i=0}^{\sigma_*-1}(\sigma_*-i)P_i(\rho_{k-1}(t),\eta_{k-1}(t)),\\
    & \frac{d}{dt} \eta_{k}(t) 
    = 2\, c- \frac{c}{\sigma_*}\, \sum\limits_{i=0}^{\sigma_*-1}\,(\sigma_*-i)P_i(\rho_{k-1}(t),\eta_{k-1}(t))
    -\frac{c}{\sigma_*}\, \sum\limits_{i=0}^{\sigma_*-1}\,(\sigma_*-i)\,(2\,\rho_{k}+1-2\,i) P_i(\rho_{k}(t),\eta_{k}(t)).
\end{align}
\end{subequations}

The equations in \eqref{eq:rho1-eta1-closure} and \eqref{eq:rhok-etak-closure}, along with the ansatz $P_n$ defined in \eqref{eq:P}, form a closed system that approximates the time evolution of the first two moments
of the stochastic process $\bsigma(t)$. Moreover, the right-hand sides of the equations \eqref{eq:rho1-eta1-closure} and \eqref{eq:rhok-etak-closure} involve only $O(\sigma_*)$ terms in the summation; thus the system is computationally efficient to evaluate.

\subsection{Analysis of the Deterministic System}
\label{sec3.2}
In this section, we prove that the deterministic system 
\eqref{eq:rho1-eta1-closure}, \eqref{eq:rhok-etak-closure}
is well-posed and results in physically admissible solutions. 
In particular, the main results in this section are Theorems 
\ref{th1} (existence and uniqueness of solutions) and \ref{th2}
(invariance of the domain $D$ defined below).

Define the domain $D = \{(\rho,\eta):0\leq\rho\leq\eta\}$ and 
let $\operatorname{int} D = \{(\rho,\eta):0<\rho <\eta\}$ denote its interior.  Define the dynamics $ y' = F(y)$
with $y = (\rho,\eta)^{\top}$ and
$F: D  \rightarrow\R^2$ given by 
\begin{equation}\label{eq:Fdef}
F(\rho,\eta) 
= \begin{pmatrix}
  F_1(\rho,\eta) \\
  F_2(\rho,\eta)
\end{pmatrix} 
=
\begin{pmatrix}
c_0(t)-c+\frac{c}{\sigma_*}\sum_{i=0}^{\sigma_*-1}(\sigma_*-i)P_i(\rho,\eta)\\
c_0(t)+c-\frac{c}{\sigma_*}\sum_{i=0}^{\sigma_*-1}(2\rho-2i+1)(\sigma_*-i)P_i(\rho,\eta),
\end{pmatrix}
\end{equation}
where
\begin{equation}\label{eq:P0def}
P_0(\rho,\eta)=\begin{cases}
1, & \text{ if } 0=\rho \leq \eta,\\
p^r, & \text{ if } 0<\rho<\eta,\\
e^{-\rho}, & \text{ if } 0<\rho=\eta,\\
\end{cases}
\end{equation}
and for $n>0$ 
\begin{equation}\label{eq:Pndef}
P_n(\rho,\eta)=\begin{cases}
0 & \text{ if } 0=\rho \leq \eta\\
\big[(1-p)^np^r\big]\frac{1}{n!} \prod_{i=1}^n (n+r-i) & \text{ if } 0<\rho<\eta\\
\frac{\rho^n}{n!}e^{-\rho} & \text{ if } 0<\rho=\eta\\
\end{cases} ,
\end{equation}
with $p=\hat{p}(\rho,\eta) = \frac{\rho}{\eta}$ and $r = \hat{r}(\rho,\eta) = \frac{\rho^2}{\eta-\rho}$ as in \eqref{eq:hat-p-hat-r}. Standard results from the theory of ODEs establish that the initial value problem $y'=F(y), ~y(0)=(\rho_0,\eta_0)^T$ will have a unique solution for $(\rho_0,\eta_0)^T\in \operatorname{int}D$.  However, in order to establish the existence of unique solutions for $(\rho_0,\eta_0)\in D,$ we extend $F$ to a function $\bar{F}$ that is defined on all of $\R^2$ and satisfies necessary conditions for the existence of a unique solution.  We also show that $D$ is invariant under the dynamics of $\bar{F}$.
%


\subsubsection{Preliminary Background Results}
For the ODE $y'=f(t,y)$ with initial condition $y(a)=y_0$, define the sequence $a=t_0<t_1<\dots t_N=T$ with nonuniform cell sizes $h_i=t_{i+1}-t_i$ and $h=\max h_i$.  The Euler polygonal arc is
\begin{subequations}
\begin{alignat}{3}
  y^h(t)   &= y_i+(t-t_i)f(t_i,y_i),&& \qquad t\in[t_i,t_{i+1}], \\
  y_{i+1}  &= y^h(t_{i+1}),&& \qquad i = 0,\dots,t_{N-1}.
\end{alignat}
\end{subequations}
The next theorem establishes the existence and uniqueness of solutions using 
Euler polygonal arcs.
\begin{theorem}\label{thm:EandU}(Theorem 7.3 of \cite{HaiNorWan00})
  Let $f$ be continuous, uniformly bounded by $M$, and Lipschitz on $\{(t,y):a\leq t\leq T,y\in\overline{B}_R(y_0)\}$, where ${B}_R(y_0)$ is the ball of radius $R$ centered at $y_0$.  Then if $T-a\leq R/M$, the Euler polygonal arcs $\{y^h(t)\}_{h>0}$ converge uniformly to a continuous function $y(t)$ as $h\rightarrow0$. Moreover, $y(t)$ is continuously differentiable and is the unique solution to the initial value problem on $[0, T]$.
\end{theorem}

The following two lemmas about Lipschitz continuity and the projection operator are also useful.
\begin{lemma}\label{thm:lipext} (from Theorem 9.58 of \cite{RocWet98})
For any $X\subset\R^n$ and any $F:X\rightarrow\R^m$ that is Lipschitz continuous relative to $X$ with constant $K$, $F$ has a unique continuous extension to $\bar{X}$, necessarily Lipschitz continuous with constant $K$.
\end{lemma}
\begin{lemma}\label{thm:proj} (Theorem 2 of \cite{CheGol59})
	Let $C\subset \R^n$ be a non-empty, closed, convex set; let $x,y \in \R^n$; and let $\proj_C(x)$ be the standard projection (in the 2-norm) of $x$ onto $C$.  Then $\|\proj_C(x)-\proj_C(y)\|\leq \|x-y\|$.
\end{lemma}

\subsubsection{Existence of unique solutions}
We establish the existence of unique solutions for the ODE system defined by the dynamics in \eqref{eq:Fdef}.
To do so, we will first show that an extension of $F$ satisfies the hypotheses of
\Cref{thm:EandU}.
\begin{lemma}\label{lem:PbndLip}
	For each $n\geq0,$ $P_n$ is locally bounded and locally Lipschitz on the interior of $D$.
\end{lemma}

The proof of \Cref{lem:PbndLip} uses the following lemma, which is proven in \cite{rothkopforen1979}.
\begin{lemma}\label{lem:P0bd}
For all $\rho \in \mathbb{R}$ and $\eta \in \mathbb{R}$ such that $0 < \rho \le \eta$ and for all integers $n \geq 0$,
\begin{equation}
    \label{inequality}
    0 \leq e^{-\rho} \le P_0(\rho,\eta) \le e^{-\rho^{2}/\eta} \le 1.
\end{equation}
\end{lemma}



It will be useful in the proof for \Cref{lem:PbndLip} and subsequent proofs to use the auxiliary function
\begin{equation}
\label{eq:Qn}
\begin{split}
   Q_n(\rho,\eta) := \frac{ P_n(\rho,\eta)}{ P_0(\rho,\eta)} 
   & = \frac{(1-p)^n}{n!} \prod_{i=1}^n (n-i+r)
    = \frac{1}{n!}\prod_{i=0}^{n-1} \Big( (1-p)i + (1-p)r \Big) =
      \frac{1}{n!}\prod_{i=0}^{n-1} \Big( (1-\frac{\rho}{\eta})i + \frac{\rho^2}{\eta} \Big)
\end{split}
\end{equation}
for $n>0$ and $(\rho,\eta) \in \operatorname{int} D$.

\begin{proof}[Proof of \Cref{lem:PbndLip}]
	Let $(\rho_0,\eta_0)\in\text{int}D=\{0<\rho<\eta\}$ and let $\delta>0$ be such that 
    ${B}_\delta\big((\rho_0,\eta_0)\big)\subset \text{int}D$. 
	According to \cref{lem:P0bd}, $P_0$ is uniformly bounded on $B_\delta\big((\rho_0,\eta_0)\big)$.
	To establish the local boundedness of $P_n$ for $n>0$, we bound $Q_n$.
	On ${B}_\delta\big((\rho_0,\eta_0)\big)$, $0<\rho/\eta<1$ and $0<\rho<\rho_0+\delta$; this means $0\leq \frac{\rho^2}{\eta}+i(1-\frac{\rho}{\eta})\leq \rho+i\leq \rho_0+\delta+n$ for $0\leq i\leq n-1$.  That, in turn, implies
	\begin{equation}
    \label{eq:Qn-bound}
	|Q_n(\rho,\eta)|\leq M_{\rho_0,\delta,n}:=\frac{(\rho_0+\delta+n)^n}{n!},
	\end{equation}
	thus providing a bound for $Q_n$ on ${B}_\delta\big((\rho_0,\eta_0)\big)$.
	
	We next show that there is a Lipschitz constant $L>0$  for $P_n$ on ${B}_\delta\big((\rho_0,\eta_0)\big)$.
	Since $P_n=Q_nP_0$, it follows from the triangle inequality that
	\begin{equation}\begin{split}
	|P_n(\rho,\eta)-P_n(\tilde{\rho},\tilde{\eta})|
	&\leq |Q_n(\rho,\eta)| |P_0(\rho,\eta)-P_0(\tilde{\rho},\tilde{\eta})|
    +|P_0(\tilde{\rho},\tilde{\eta})| |Q_n(\rho,\eta)-Q_n(\tilde{\rho},\tilde{\eta})|\\
	&\leq M_{\delta,n}|P_0(\rho,\eta)-P_0(\tilde{\rho},\tilde{\eta})|
	+|Q_n(\rho,\eta)-Q_n(\tilde{\rho},\tilde{\eta})|,	\end{split}\end{equation}
	where
	$P_0(\rho,\eta) \le 1$ by
	\Cref{lem:P0bd}.
	Thus it is sufficient to show that $P_0$ and $Q_n$ are each Lipschitz in $B_\delta\big((\rho_0,\eta_0)\big)$.
	To do so, we bound the gradient of each on $\operatorname{int} D$.
	First, we consider $P_0$. To simplify the gradient computations, we rewrite \eqref{eq:P0def}, for $0 < \rho < \eta$,
 \begin{equation}
    P_0(\rho,\eta) = \left(\frac{\rho}{\eta}\right)^{\rho^{2}/(\eta-\rho)} = 
    \exp\left(\frac{\rho^{2}}{\eta-\rho}\,\log\left(\frac{\rho}{\eta}\right)\right) = 
    \exp\left({-\frac{p\,\rho}{p-1}\,\log(p)}\right).
 \end{equation}
From this formula, it is easy to show that
	\begin{subequations}
    \label{eq:P0_partials}
    \begin{align}
	\partial_\rho P_0(\rho,\eta) &= \left(\left[\frac{2\rho\eta-\rho^2}{(\eta-\rho)^2}\right]\log\left(\frac{\rho}{\eta}\right)+\frac{\rho}{\eta-\rho} \right)P_0(\rho,\eta),\\
	\p_\eta P_0(\rho,\eta) &= \left(\left[\frac{-\rho^2}{(\eta-\rho)^2}\right]\log\left(\frac{\rho}{\eta}\right)-\frac{\rho^2}{(\eta-\rho)\eta} \right)P_0(\rho,\eta).
	\end{align}\end{subequations}
	Because $ 0 < \rho < \eta$ on $B_\delta\big((\rho_0,\eta_0)\big)$, the partial derivatives in \eqref{eq:P0_partials} are continuous on $B_\delta\big((\rho_0,\eta_0)\big)$. Expanding $\p_\rho P_0$ using a Taylor expansion for $x \mapsto \log(1-x)$, along with the bound $0\leq P_0(\rho,\eta)\leq 1$,   gives 
	\begin{equation}\begin{split}
    \label{eq:dP0-drho}
     	|\p_\rho P_0(\rho,\eta) |
      &\leq \Big\lvert-\frac{2\rho\eta-\rho^2}{(\eta-\rho)^2}\Big(\frac{\eta-\rho}{\eta}+\frac{(\eta-\rho)^2}{2\eta^2}+\sum_{k=3}^\infty\frac{(\eta-\rho)^k}{k\eta^k}\Big)+\frac{\rho}{\eta-\rho}\Big\rvert\\
	    &=\Big\lvert -\frac{2\rho\eta-\rho^2}{\eta(\eta-\rho)}-\frac{2\rho\eta-\rho^2}{2\eta^2}+\frac{\rho}{\eta-\rho}-\frac{2\rho\eta-\rho^2}{(\eta-\rho)^2}\sum_{k=3}^\infty\frac{(\eta-\rho)^k}{k\eta^k}\Big\rvert\\
	    &= \Big\lvert-2\frac{\rho}{\eta}+\frac{\rho^2}{2\eta^2}-\frac{2\rho\eta-\rho^2}{\eta^2}\sum_{k=3}^\infty\frac{(\eta-\rho)^{k-2}}{k\eta^{k-2}}\Big\rvert\\
	    & =\Big\lvert -2\frac{\rho}{\eta}+\frac{\rho^2}{2\eta^2}
     -\left(2\frac{\rho}{\eta}-(\frac{\rho}{\eta})^2 \right)\sum_{k=1}^\infty \frac{1}{k+2}(1-\frac{\rho}{\eta})^k \Big\rvert\
     =\Big\lvert -2\frac{\rho}{\eta}+\frac{\rho^2}{2\eta^2}
     - f\left(\frac{\rho}{\eta}\right)\Big\rvert,
\end{split}\end{equation}
	where $f(x) = (2x-x^2)\sum_{k=1}^\infty\frac{1}{k+2}(1-x)^k$ for $x\in(0,1)$. On $\operatorname{int}D$, $|-2\frac{\rho}{\eta}+\frac{\rho^2}{2\eta^2}|$ is bounded by 2.5. Moreover, for $x\in (0,1)$, $f(x)$ is non-negative and 
	\begin{equation}
	    f(x)\leq (2x-x^2)\sum_{k=1}^\infty\frac{(1-x)^k}{k}=(2x-x^2)\sum_{k=1}^\infty (-1)\frac{(-1)^{k-1}(x-1)^k}{k}=-(2x-x^2)\log x.  
	\end{equation}
    The mapping $x \mapsto -(2x-x^2)\log x$  is a strictly concave, positive  function on $(0,1)$ with a coarse upper bound of $1$; numerically, the bound is closer to $0.615.$ Thus $f(\frac{\rho}{\eta})\leq 1$ on $\operatorname{int}D$, which in light of \eqref{eq:dP0-drho}, implies that $|\p_\rho P_0|\leq 3.5$ on $\operatorname{int}D$.
    Similarly,
    \begin{equation}\begin{split}
	    |\p_\eta P_0(\rho,\eta)| &\leq \Big\lvert-\frac{\rho^2}{(\eta-\rho)^2}\Big(\frac{\eta-\rho}{\eta}+\frac{(\eta-\rho)^2}{2\eta^2}+\sum_{k=3}^\infty\frac{(\eta-\rho)^k}{k\eta^k}\Big)-\frac{\rho^2}{\eta(\eta-\rho)}\Big\rvert\\
	    &= \Big\lvert\frac{\rho^2}{2\eta^2}+\frac{\rho^2}{\eta^2}\sum_{k=3}^\infty \frac{(\eta-\rho)^{k-2}}{k\eta^{k-2}}\Big\rvert
     = \Big\lvert\frac{\rho^2}{2\eta^2}+\frac{\rho^2}{\eta^2}\sum_{k=1}^\infty \frac{(\eta-\rho)^k}{(k+2)\eta^k}\Big\rvert 
     = \Big\lvert\frac{\rho^2}{2\eta^2} + g\left( \frac{\rho}{\eta}\right)\Big\rvert,
	\end{split}\end{equation}
	where $g(x) = x^2\sum_{k=1}^\infty \frac{1}{k+2}(1-x)^k$ is non-negative for $x\in(0,1)$. On $\operatorname{int} D$, $\frac{\rho^2}{2\eta^2}\leq 0.5$.  
	Additionally, $g(x)\leq x^2\sum_{k=1}^\infty \frac{(1-x)^k}{k} = -x^2\log x$, and the  
	the mapping $x \mapsto -x^2\log x$ is positive and bounded above on $(0,1)$ by $\frac{1}{2 e}$.
	Thus $g(\frac{\rho}{\eta})\leq \frac{1}{2 e}$, which implies that $|\p_\eta P_0|\leq 0.5(1+e)$ on $\operatorname{int} D$.  
	
	Having established that the partial derivatives of $P_0$ are uniformly bounded on $\operatorname{int} D$, we next consider the derivatives of $Q_n$.
	The function $Q_n$ is a finite sum of terms of the form $c_{k,\ell}\frac{\rho^k}{\eta^\ell}$ with constants $c_{k,\ell}$ (independent of $\rho,\eta$) and positive integers $k,\ell$ such that $\ell+1\leq k\leq 2\sigma_*.$ 
	Since $0\leq \rho/\eta \leq 1$, the gradient of $\frac{\rho^k}{\eta^\ell}$ satisfies the bound
	\begin{equation}
	 \left \|\nabla \frac{\rho^k}{\eta^\ell} \right \| 
        = \sqrt{\left(\frac{k\rho^{k-1}} {\eta^\ell}\right)^2+\left(\frac{\ell\rho^{k}}{\eta^{\ell+1}}\right)^2} 
        \leq \rho^{k-(\ell+1)}
            \sqrt{k^2 +\ell^2}.
	\end{equation}
	Thus, since $\ell+1\leq k\leq 2\sigma_*$, $\|\nabla Q_n\|$ is  bounded on $B_\delta(\rho_0,\eta_0)$ by a constant $K_{\rho_0,\delta,\sigma_*}$ that depends on $\sigma_*$, $\rho_0$, and $\delta$.
	As a consequence,
	\begin{equation}\begin{split}
	|P_n(\rho,\eta)-P_n(\tilde{\rho},\tilde{\eta})|
	&\leq M_{\rho_0,\delta,n}K_1\|(\rho,\eta)-(\tilde{\rho},\tilde{\eta})\|
	+K_{\rho_0,\delta,\sigma_*}\|(\rho,\eta)-(\tilde{\rho},\tilde{\eta})\|,
	\end{split}\end{equation}
	where the constant $K_1$ is the Lipschitz constant for $P_0$ on $\operatorname{int} D$.
\end{proof}
\begin{lemma}
	For each $n=0,1,\dots,\sigma_*-1$, $P_n$ is locally Lipschitz and locally bounded on $D = \{(0\leq\rho\leq\eta\}$.
\end{lemma}
\begin{proof}
	We first show that $P_n$ is locally bounded. Let $(\rho_0,\eta_0)\in D$ and consider $B_\delta(\rho_0,\eta_0)$ for some $\delta>0$.  Then it follows from the definitions in \eqref{eq:P0def} and \eqref{eq:Pndef} that, on  $N =B_\delta(\rho_0,\eta_0)\cap  D$,
 \begin{equation}
 \label{eq:Pn-bound}
     |P_n| 
     \leq \max \left\{
     \sup_{N\cap\{0 = \rho \leq \eta \}} P_n,
     \sup_{N\cap \operatorname{int} D} P_n,
     \sup_{N\cap\{0 \leq \rho =\eta\}}e^{-\rho}\frac{\rho^n}{n!} \right\}.
 \end{equation}
The first term in \eqref{eq:Pn-bound} is bounded by 1; cf. \eqref{eq:P0def} and \eqref{eq:Pndef}. By the definition of $Q_n$ in \eqref{eq:Qn} $|P_n| \leq |P_0 Q_n| \leq M_{\delta,n}$ on ${N\cap \operatorname{int} D}$, where $M_{\delta,n}$ is given in \eqref{eq:Qn-bound}.  Thus the second term in \eqref{eq:Pn-bound} is bounded. Meanwhile, the third term is bounded on $\bar{N}\cap\{\rho = \eta\}$ due to the continuity of $e^{-\rho}\frac{\rho^n}{n!}$.   
	
	We next show that $P_n$ is locally Lipschitz. By \Cref{lem:PbndLip}, $P_n$ is Lipschitz on $N\cap \operatorname{int} D$. The continuous extension on the sets $N\cap\{0 = \rho \leq \eta \}$ and $N\cap\{0 \leq \rho =\eta\}$ to the values defined in \eqref{eq:P0def} and \eqref{eq:Pndef} is established in \eqref{eq:nb-alt}--\eqref{eq:ray2}.
	Thus $P_n|_{N\cap D}$ is a continuous extension of a Lipschitz continuous function on $N\cap\text{int}D$ and is necessarily unique and Lipschitz (e.g., \Cref{thm:lipext}). This establishes that $P_n$ is locally Lipschitz and locally bounded on $ D$.
\end{proof}
\begin{lemma}
	Let $c_0(t)\equiv0$.  The function $\bar{F}(\rho,\eta) := F(\proj_{ D}(\rho,\eta))$ is a locally bounded, locally Lipschitz function on $\R^2$.
\end{lemma}
\begin{proof}
	The entries of $F$ (see \eqref{eq:Fdef}) are linear combinations of locally bounded, locally Lipschitz functions, and thus $F$ is as well.
	Because $ D$ is a convex set, the projection of $(\rho,\eta)$ onto $ D$ is unique. Hence local boundedness of $\bar{F}$ is inherited by local boundedness of $F$. Moreover, since projection onto a convex set is non-expansive (\Cref{thm:proj}), it follows that $\bar{F}$ is also locally Lipschitz, i.e., 
	\begin{equation}
		|\bar{F}(\rho_0,\eta_0)-\bar{F}(\rho_1,\eta_1)|\leq K\|(\rho_0,\eta_0)-(\rho_1,\eta_1)\|
	\end{equation}
	where $K$ is the Lipschitz constant for $F$ on a neighborhood containing $\proj_{ D}(\rho_0,\eta_0)$ and $\proj_{ D}(\rho_1,\eta_1)$.
\end{proof}
We have now established that $\bar{F}$ satisfies the hypotheses in \Cref{thm:EandU}, and thus we can immediately conclude the following.
\begin{theorem}
\label{th1}
	Let $c_0(t)$ be uniformly bounded and Lipschitz for $t\in[0,T]$ for some $T>0$.  Then, for any initial condition $(\rho_0,\eta_0)\in D$, there is a solution to the initial value problem	
	\begin{equation}
    \label{eq:ivp}
	\frac{d}{dt}\begin{pmatrix} \rho(t)\\\eta(t)\end{pmatrix}=\bar{F}(\rho(t),\eta(t)),\quad (\rho(0),\eta(0)) = (\rho_0,\eta_0).
	\end{equation}
\end{theorem}

\subsubsection{Invariance of the domain}
We show in the next theorem that trajectories that originate in $D$ remain there.
\begin{theorem}
\label{th2}
	Assume for all $t\geq 0$, that $c_0(t)\geq 0$ and that $(\rho(t),\eta(t))$ solves the initial value problem \eqref{eq:ivp} with initial data $(\rho(0),\eta(0)) \in D$.  Then $(\rho(t),\eta(t))\in D$ for all $t\geq0$.  
\end{theorem}
\begin{proof}
	Because $D$ is a closed convex set, it is sufficient to show that for any $(\rho,\eta)\in\p D$ 
 \begin{equation}
     \langle\xi,\bar{F}(\rho,\eta)\rangle\leq0, \quad \forall \xi \in N_D (\rho,\eta),
 \end{equation}
 where $N_D (\rho,\eta)$ is the normal cone of $D$ at $(\rho,\eta)$.  Then by \cite[Chapter 5, Theorem 7]{AubCel84}, $(\rho(t),\eta(t))\in D$.  There are three cases
 \begin{itemize}
 \item{\bf  Case 1: $\eta=\rho>0$.} The normal cone for such $(\rho,\eta)$ is given by
    \begin{equation}
     N_D(\rho,\eta) = \{ \xi \in \R^2 :  \xi = a \xi_1\,, a \geq 0 \},
    \end{equation}
    where $\xi_1=(1,-1)^T$.  Thus, we need only verify that 
    \begin{equation}
    \label{eq:F2-F1}
        F_2(\rho,\eta)-F_1(\rho,\eta) = 2c-\frac{2c}{\sigma_*}\sum_{i=0}^{\sigma_*-1}(\rho-i+1)(\sigma_*-i)P_i(\rho,\eta)\geq 0.
    \end{equation}
    Since $\eta=\rho>0$, $P_i(\rho, \eta) = e^{-\rho}\frac{\rho^i}{i!}$. Hence, according to Proposition \ref{prop:detailed_calc} in the Appendix
    \begin{equation}
      \sum_{i=0}^{\sigma_*-1}(\rho-i+1)(\sigma_*-i)P_i(\rho,\eta)  
      = \s \sum_{i=0}^{\sigma_*} P_i(\rho,\eta) 
    \end{equation}
    and so $F_2(\rho,\eta)-F_1(\rho,\eta)=2c-2c\sum_{i=0}^{\sigma_*}P_i(\rho,\eta)$.  Moreover, since  $\sigma_*$ is finite and $i \mapsto P_i(\rho,\eta)$ is a probability distribution, $\sum_{i=0}^{\sigma_*}P_i(\rho,\eta)<1$ , which implies $F_2(\rho,\eta)-F_1(\rho,\eta)>0$. 

\item{\bf Case 2: $\eta > \rho=0$}.  The normal cone for such $(\rho,\eta)$ is given by
    \begin{equation}
     N_D(\rho,\eta) = \{ \xi \in \R^2 :  \xi = b \xi_2\,, b \geq 0 \},
    \end{equation}
 where $\xi_2=(-1,0)^T$.  Thus we need to verify that $F_1(\rho,\eta) \geq 0$, which clearly holds since $F_1(0,\eta)= c_0 \geq 0$.
 
 \item{\bf Case 3: $\eta = \rho =0$}.  In this case,
     \begin{equation}
     N_D(\rho,\eta) = N_D(0,0) = \{ \xi \in \R^2 :  \xi = \lambda \xi_1+(1-\lambda)\xi_2, \quad \text{for some } \lambda \in [0,1]\}.
    \end{equation}
Thus we need to verify that
\begin{equation}
    \lambda(F_1 - F_2) - (1-\lambda)F_1 \leq 0, \quad \forall \lambda \in [0,1]
\end{equation}
From \eqref{eq:F2-F1}, it follows that $F_1(0,0)-F_2(0,0)=0$ and  $F_1(0,0)=c_0 \geq 0$.  Hence
\begin{equation}
    \lambda(F_1 - F_2) - (1-\lambda)F_1 
    = - (1-\lambda)c_0 \leq 0.
\end{equation}
 \end{itemize}

\end{proof}

\begin{remark}
    The results established in \Cref{th1} and \Cref{th2} apply only to the dynamics of $(\rho_1,\eta_1)$.  However, for general $k$, the inflow terms (analogous to $c_0$) on the the right-hand side of \eqref{eq:rhok-etak-closure} are simply linear combinations of $P_0(\rho_{k-1}(t),\eta_{k-1}(t))$ and $P_n(\rho_{k-1}(t),\eta_{k-1}(t))$. Thus, the results in \Cref{th1} and \Cref{th2} can be extended to the dynamics of $(\rho_k,\eta_k)$, $k >1$, by induction on $k$.
\end{remark}

\begin{remark}
    These results additionally mean that for sufficiently small time steps, we expect that a forward Euler method will produce solutions which remain in $D$ and thus correspond to realizeable mean and variance pairs. As a result, we use in \Cref{sec5}, such a method for our approximation results.
\end{remark}


\section{General multi-stage queue with stage-dependent throttling}
\label{sec4}
Many of the results above can be extended to a more realistic case when $\sigma^{(k)}_*$ depends on the stage $k$. In this setting the throttling function \eqref{v} becomes,
\begin{equation}
\label{v_k}
v_k(\sigma_k)=\begin{cases} 
      0 & \sigma_k \leq 0, \\
      \frac{\sigma_k}{\sigma^{(k)}_{*}} & 0\leq \sigma_k \leq \sigma^{(k)}_{*},\\
      1 & \sigma_k \geq \sigma^{(k)}_{*}.
   \end{cases}
\end{equation}
The extension of \Cref{prop1} to the case where we use $v_k(\sigma_k)$ for the throttling function is given in \Cref{thm:Markovext}, which we state without proof. 
 \begin{theorem}
\label{thm:Markovext}
Let $(\sigma_*^{(k)})_{k=1}^N$ be an arbitrary sequence of natural numbers. The continuous-time Markov chain $\bsigma(t)$ which uses $v_k(\sigma_k)$ for its throttling function has a unique
limiting distribution given by 
\begin{equation}
\pi(\balpha) = C\,\frac{ c_0^{\alpha_1}}{c^{\alpha_1} \prod_{k=1}^{\alpha_1}v_1(k) } \times \frac{ c_0^{\alpha_2}}{c^{\alpha_2}  \prod_{k=1}^{\alpha_2}v_2(k) } \times \cdots \times \frac{ c_0^{\alpha_N}}{c^{\alpha_N} \prod_{k=1}^{\alpha_N}v_N(k) } ,
\label{eq:new stationery}
\end{equation}
where $C$ is the normalization constant given by
\begin{equation}\begin{split}
C&= \left( {\sum_{\alpha_1,\ldots,\alpha_N} \frac{ c_0^{\alpha_1}}{c^{\alpha_1} \prod_{k=1}^{\alpha_1}v_1(k) } \times \frac{ c_0^{\alpha_2}}{c^{\alpha_2} \prod_{k=1}^{\alpha_2}v_2(k) } \times \cdots \times \frac{ c_0^{\alpha_N}}{c^{\alpha_N}  \prod_{k=1}^{\alpha_N}v_N(k) }} \right)^{-1},
\end{split}\end{equation}
and $\prod_{k=1}^{\alpha_j}v_j(k) \equiv 1$ if $\alpha_j=0$.\\
\end{theorem}

We then use the negative binomial approximation introduced in 
\eqref{eq:46}
to find a closed form for the right-hand side of equations \eqref{eq:rho1-eta1-closure} and \eqref{eq:rhok-etak-closure} arising from \eqref{v_k} acting as the throttling function.  Finally, we conclude this section with the following theorem, which is a generalization of \Cref{th2}. The proofs needed are essentially unchanged.

\begin{theorem}
\label{th4}
Assume $c_0(t) > M > 0$ is a piecewise continuous function and that the initial condition is such that $\eta_k(0)\ge\rho_k(0)\geq 0.$ Let $(\sigma^{(k)}_*)_{k=1}^{N}$ be any arbitrary sequence of natural numbers for the throttling thresholds.  Then there is a unique solution pair $(\rho_k(t),\eta_k(t))$ to the differential equations analogous to \eqref{eq:rho1-eta1-closure} and \eqref{eq:rhok-etak-closure} when using \eqref{v_k} to define the stage-dependent throttling functions. Moreover, 
\begin{equation}
    \rho_k(t) \geq 0 \qquad \text{and} \qquad \eta_k(t) \ge \rho_k(t),
\end{equation}
for all $k\ge1$ and all $t$. 
\end{theorem}

\section{Numerical Simulations}
\label{sec5}

In this section, we demonstrate  numerically the validity of the approximate moment system
\eqref{eq:rho1-eta1-closure}  and \eqref{eq:rhok-etak-closure} for the mean and variance of the multi-stage queue $\bsigma(t)$. 
To this end, we compare the solutions of \eqref{eq:rho1-eta1-closure}  and \eqref{eq:rhok-etak-closure} with the mean and variance of Monte-Carlo simulations for the stochastic process 
$\bsigma(t)$.  
All simulations use a zero initial condition, which corresponds to an empty multi-stage queue.

Many parameter regimes have been studied, including constant and time-dependent input rates $c_0$ and different values of the thresholding parameter $\sigma_*$.
Below we present a few of the most interesting cases; a more exhaustive numerical investigation can be found in 
\cite{daneshvar2021}.

We consider between $100$ and $300$ stages in our simulations and use the time step $\dt=0.001$ to simulate the stochastic model. We generate $M=10^5$ Monte-Carlo realizations of the Markov chain $\bsigma(t)$ for the different time-slices. 
Then we compute 
\begin{equation}
    \begin{split}
        \mathbb{E}_k(t) \approx \frac1M \sum\limits_{p=1}^{M} \sigma^{(p)}_{k}(t)
        \quad
        \quand
        \quad
        \mathbb{V}_k(t) \approx \frac1M \sum\limits_{p=1}^{M} (\sigma^{(p)}_{k}(t)-\mathbb{E}_k(t))^2,
    \end{split}
\end{equation}
to approximate the mean and variance of the distribution for each stage $k$ and time $t$.
Here $\sigma^{(p)}_{k}(t)$, $k=1,\ldots,N$ is the $p$-th realization.
We utilize the variable time-step Euler method in Python by using Numpy and Numba libraries for the integration
of the deterministic system in \eqref{eq:rho1-eta1-closure}  and \eqref{eq:rhok-etak-closure}.  
The deterministic model is approximately one order of magnitude faster than the Monte Carlo simulations of the stochastic model.

\subsection{Piece-wise Constant Input Rate}
\label{num1}

We first consider the piecewise constant input rate
\begin{equation}
\label{input1}
c_0(t)=\begin{cases} 
      6 & t \leq 30, \\
      0 & t > 30.
   \end{cases}
\end{equation}
We set $c=10$ and consider two values of the throttling parameter:  $\sigma_*=3$ and $\sigma = 5$. \itc{The initial condition for Monte-Carlo simulations is $\sigma_k(0)=0$ and the corresponding initial condition of the deterministic system \eqref{eq:rho1-eta1-closure}, \eqref{eq:rhok-etak-closure}
is $\rho_k(0) = \eta_k(0) = 0$ for $k=1,\ldots,N$.}
Figures \ref{fig:p1} and \ref{fig:p3} depict, at selected times,
solutions of $\bE_k$ and $\rho_k$ when $\sigma_*=3$ and $\sigma_*=5$, respectively. Figures \ref{fig:p2} and  \ref{fig:p4} depict solutions for $\mathbb{V}_k$ and  $\eta_k$ under similar condition conditions.
\itc{The dynamics of the stochastic model and the corresponding deterministic system resemble a nonlinear wave propagation problem where the information enters the system on the left (stage $k=0$) and then moves through all stages. The leading front corresponds to the information entering at time $t=0$, and the trailing front emerges after $t>30$ because the input is stopped in \eqref{input1} at time $t=30$.}

The behavior of $\bE_k$ in stochastic simulations is reproduced very well by the deterministic model \eqref{eq:rho1-eta1-closure}-\eqref{eq:rhok-etak-closure}. Both leading and trailing front propagation are captured correctly
for $\bE_k$.  In particular, the locations of both fronts are captured very accurately.  The trailing front is slightly more diffusive in stochastic simulations, which is common in many stochastic lattice models and their deterministic counterparts. In general, stochastic lattice models tend to be more diffusive compared to deterministic closure models due to the inherent
randomness of stochastic dynamics. (See for example the discussion about asymmetric simple exclusion process (ASEP) models and noisy PDEs in \cite{helbing2001,css00}).
Because the trailing edges of the waves are much sharper than the leading fronts, the diffusive nature of the stochastic simulations tends to be more dominant there. 
The behavior of the variance $\mathbb{V}_k$ is also well-reproduced by the deterministic model, although the discrepancies tend to be larger than those of the mean $\mathbb{E}_k$ \itc{even for shorter times
(this is particularly evident at time $t=10, 20$ in Figures \ref{fig:p2} and \ref{fig:p4}).}
Nonetheless, the discrepancies in $\eta_k$ do not significantly affect the behavior of $\rho_k$ in this model.
\begin{figure} \centering
\includegraphics[scale=0.45]{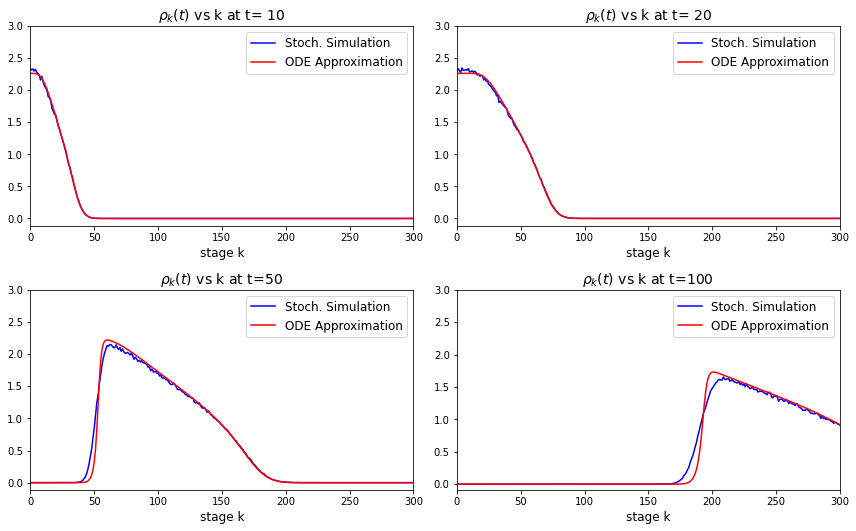}
\caption{$\bE_k$ and $\rho_k$ vs $k$ with $c=10$, $\sigma_*=3$ and piece-wise $c_0(t)$, for $300$ stages and times $t=10,20,50,100$ }
\label{fig:p1}
\end{figure} 
\begin{figure} \centering
\includegraphics[scale=0.45]{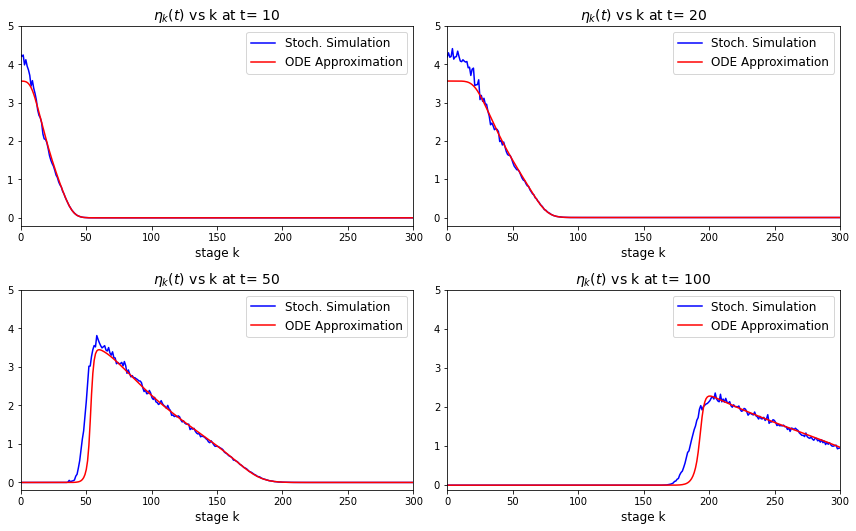}
\caption{$\mathbb{V}_k$ and $\eta_k$ vs $k$ with $c=10$, $\sigma_*=3$ and piece-wise $c_0(t)$, for $300$ stages and times $t=10,20,50,100$ }
\label{fig:p2}
\end{figure} 
\begin{figure} \centering
\includegraphics[scale=0.45]{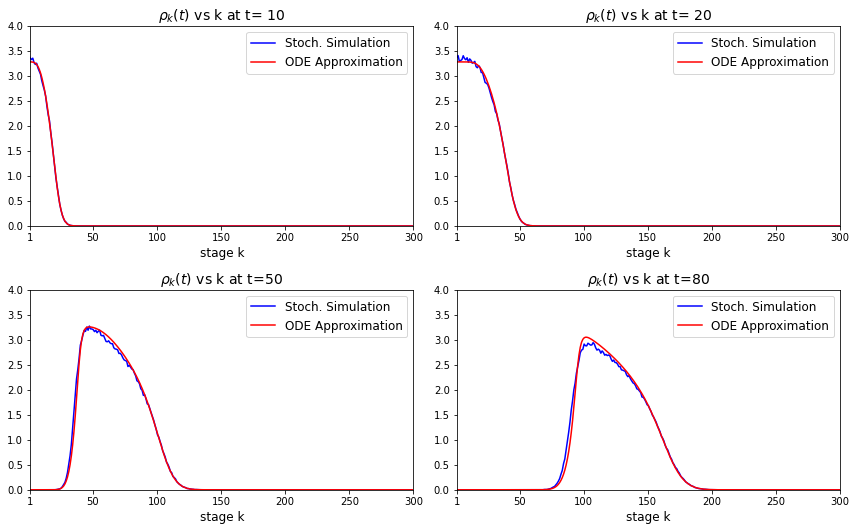}
\caption{$\bE_k$ and $\rho_k$ with $c=10$, $\sigma_*=5$ and piece-wise $c_0(t)$, for $300$ stages and times $t=10,20,50,80$ }
\label{fig:p3}
\end{figure} 
\begin{figure} \centering
\includegraphics[scale=0.45]{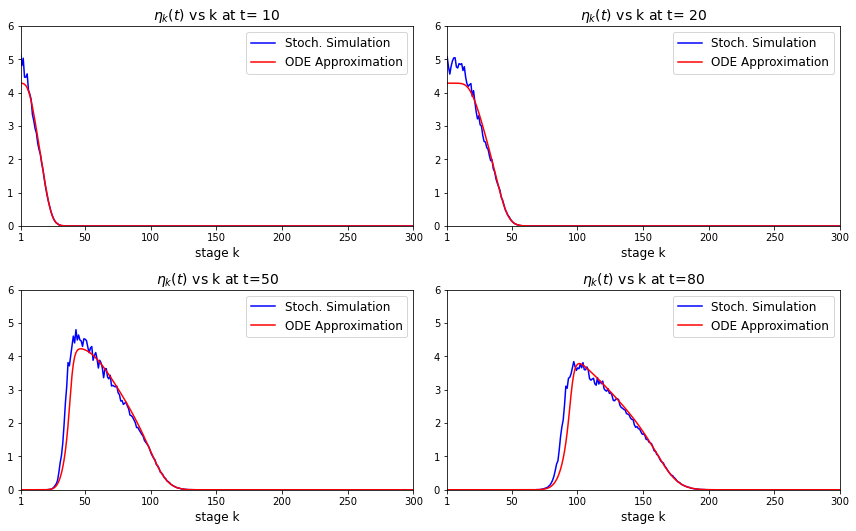}
\caption{$\mathbb{V}_k$ and $\eta_k$ vs $k$ with $c=10$, $\sigma_*=5$ and piece-wise $c_0(t)$, for $300$ stages and times $t=10,20,50,80$ }
\label{fig:p4}
\end{figure} 

For $\sigma_*=3$ and $\sigma_*=5$, both systems spend a considerable amount of time in the throttling regime, i.e., $\sigma_k(t) < \sigma_*$.  Moreover, the speed of propagation is slower for $\sigma_*=5$ compared to $\sigma_*=3$, in both the stochastic and deterministic simulations. This behavior is to expected since information is processed at a slower rate for $\sigma_*=5$.  In particular, $v(x, \sigma_*=5) < v(x, \sigma_*=3)$ for $x=1,2,3,4$.

\subsection{Piecewise-Constant Input Rate and Fluctuating Throttling Parameter}
\label{num3}

We next focus on numerical simulations of the stochastic 
and corresponding deterministic models with stage-dependent $\sigma_*^{(k)}$.
For each stage, $\sigma^{(k)}_*$ is generated from an i.i.d. discrete distribution
with $\bE[\sigma_*]=5$.
\itc{However, once generated, $\sigma^{(k)}_*$ is kept constant in all Monte-Carlo simulations.}
We then can compare the results of these simulations against those using constant $\sigma_*=5$ presented earlier in  Figures \ref{fig:p3} and \ref{fig:p4} of Section \ref{num1}.
Other parameters in the model are identical to those considered in Section \ref{num1}. In particular, the input rate is given by \eqref{input1}, and $c=10$.
We consider two different distributions for the generation of $\sigma^{(k)}_*$.
\itc{The main goal here is to understand how fluctuations of 
$\sigma^{(k)}_*$ affect the speed of propagation.}
The two samples of $\sigma^{(k)}_*$ are generated from  distributions
\begin{equation}
\label{random_sig3}
\Pr[\sigma_*=4] = 0.3, \quad 
\Pr[\sigma_*=5] = 0.4, \quad
\Pr[\sigma_*=6] = 0.3.
\end{equation}
and 
\begin{equation}
\label{random_sig4}
\Pr[\sigma_*=1] = 0.2, \quad 
\Pr[\sigma_*=2] = 0.2, \quad
\Pr[\sigma_*=6] = 0.2, \quad
\Pr[\sigma_*=8] = 0.4,
\end{equation}
respectively.
\itc{These two samples correspond to two different models which are studied separately.}
For the distribution in \eqref{random_sig3} the variance is $0.6$, and the skewness is identically zero.  For the distribution in \eqref{random_sig4} the variance is $8.8$ and the skewness is $-0.2758$.
\itc{We would like to emphasize that after the value $\sigma^{(k)}_*$ is generated (independently for each stage $k$) from one of the distributions above, it is kept constant in simulations. Therefore, effectively, $\sigma^{(k)}_*$ is not stochastic and does not change for different Monte-Carlo simulations. The Monte-Carlo averaging is performed only with respect to different paths of the stochastic system while keeping a particular sample of $\sigma^{(k)}_*$ fixed.}

\begin{figure} \centering
\includegraphics[scale=0.45]{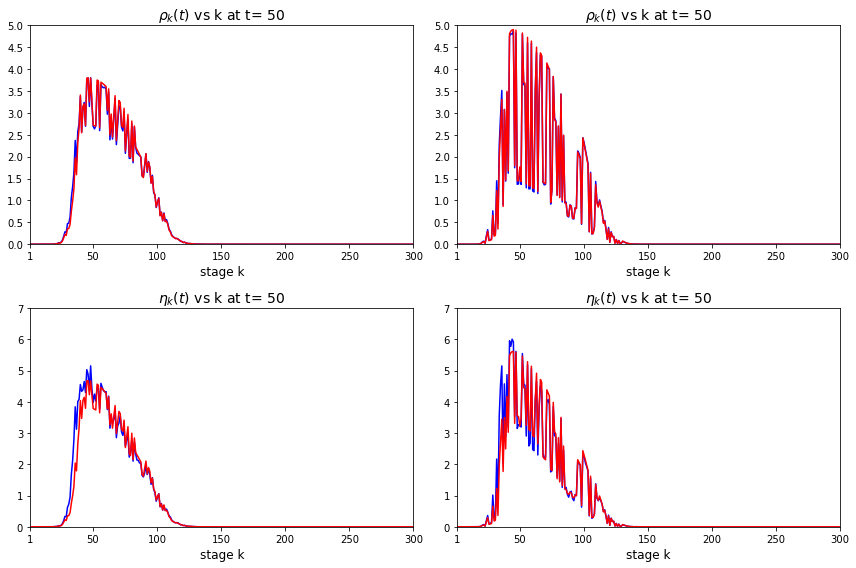}
\caption{Left: $\bE_k / \rho_k(t)$ and $\mathbb{V}_k / \eta_k(t)$ for $t=50$ and fluctuating $\sigma^{(k)}_*$ generated according to \eqref{random_sig3}. Right: $\bE_k / \rho_k(t)$ and $\mathbb{V}_k / \eta_k(t)$ for $t=50$ and fluctuating $\sigma^{(k)}_*$ generated according to \eqref{random_sig4}.}
\label{fig:p12}
\end{figure} 
The deterministic model reproduces well the behavior of the stochastic model using fluctuating $\sigma^{(k)}_*$ generated according to either
\eqref{random_sig3} or \eqref{random_sig4}.
The overall performance of the deterministic model 
is not affected by the \itc{fluctuations of $\sigma^{(k)}_*$.}
As expected, larger fluctuations in $\sigma^{(k)}_*$ yield larger  variations in $\bE_k$ and $\mathbb{V}_k$ from one stage to another. 
\itc{Since the empirical mean 
$\overline{\sigma_*} = N^{-1}\sum_{k=1}^N \sigma^{(k)}_* \approx 5$ in both cases, we can compare simulations with fluctuating $\sigma^{(k)}_*$ and numerical results  
for constant $\sigma^{(k)}_*=5$ (Figures \ref{fig:p3} and \ref{fig:p4} of Section \ref{num1}).
In particular, the speed of propagation of the leading front is not affected by fluctuations of $\sigma^{(k)}_*$.  
The results indicate that deterministic simulations with 
$\sigma^{(k)}_* = \overline{\sigma_*}$ should adequately reproduce major large-scale quantities of interest.}

\section{Conclusions}
\label{sec6}
In this paper, we have proposed and investigated a spatially-extended version of an M/M/s queue that is used to model the flow of information through the stages of a computer processor as a stochastic process.  We have analyzed the long-term behavior of this model, found its stationary distribution, and developed a closure based on the negative binomial distribution in order to approximate the evolution of the mean and variance of the stochastic process by a coupled system of deterministic ODEs.  

Since the model considered here is asymmetric, the deterministic ODE for the mean and variance of the first stage decouples from the dynamics of other stages.  We show that this ODE is well-posed and possesses a physically meaningful, invariant domain that ensures the closure remains well-posed.  These results can be extended in a straight-forward manner to the ODEs that describe the evolution of the mean and variance at consecutive stages.  

Numerical experiments have been performed to compare the stochastic process and the deterministic ODE.  Even though the latter is generally an order of magnitude faster to simulate, the results of these experiments demonstrate good agreement in the mean and variance of two models over a range of input parameters, even when the thresholding parameter at each stage is chosen randomly.  Such agreement is in contrast to the standard mean-field model approximation, which has been shown to perform poorly \cite{daneshvar2021}.  Moreover, since the ODE system keeps track of both means and variances, it can be used to quantify uncertainty in the mean.

\itc{Although the model considered here is motivated by information processing in a supercomputer, there are implications 
for general M/M/s queues as well. In particular, our results indicate that averaged quantities for the non-equilibrium behavior of M/M/s queues can be successfully approximated by a deterministic model. In addition to the mean and variance, this deterministic model can also be used to compute the mean waiting time and the mean response time for each stage, $k$ (see \cite{daneshvar2021} for details). Therefore, our results can be successfully used to model other physical systems.}

In this paper, we have focused on information flow
within a single computational unit, but the stochastic processes model can be extended to describe other multi-stage production systems. 
For example, a more realistic model of information flow should include a multi-processor architecture and proper coupling between processors.
Such a model and its analysis will be the subject of future work.


%

\appendix

\section{Appendix}

\begin{proposition} For any $\rho \geq 0$ and $\s \geq 1$,
\label{prop:detailed_calc}
        \begin{equation}
      \sum_{i=0}^{\sigma_*-1}(\rho-i+1)(\sigma_*-i) \frac{\rho^i}{i!} 
      = \sigma_* \sum_{i=0}^{\sigma_*} \frac{\rho^i}{i!} 
    \end{equation}
\end{proposition}
\begin{proof}
The proof is a direct calculation. Let
\begin{align}
S= \sum_{i=0}^{\sigma_*-1}(\rho-i+1)(\sigma_*-i)\frac{\rho^i}{i!} 
&= I - II - III + IV + V - VI
\end{align}
where
\begin{subequations}
\begin{align}
    I &= \rho \s \sum_{i=0}^{\s-1} \frac{\rho^{i}}{i!}
    = \s \sum_{i=0}^{\s-1} \frac{\rho^{i+1}}{i!}
    \\
    II &= \rho \sum_{i=0}^{\s-1} i\frac{\rho^{i}}{i!} 
    = \rho \sum_{i=1}^{\s-1} \frac{\rho^{i}}{(i-1)!} 
    = \sum_{k=0}^{\s-2} \frac{\rho^{k+2}}{k!} 
    \\
    III &= \s \sum_{i=0}^{\s-1} i\frac{\rho^{i}}{i!} 
    = \s \sum_{i=1}^{\s-1} \frac{\rho^{i}}{(i-1)!} 
    = \s \sum_{k=0}^{\s-2} \frac{\rho^{k+1}}{k!} 
    \\
    IV &=  \sum_{i=0}^{\s-1} i^2\frac{\rho^{i}}{i!} 
    = \sum_{i=1}^{\s-1} i\frac{\rho^{i}}{(i-1)!} 
    = \sum_{k=0}^{\s-2} (k+1)\frac{\rho^{k
    +1}}{k!} 
    \\
    V &= \s \sum_{i=0}^{\s-1} \frac{\rho^{i}}{i!} 
    \\
    VI &= \sum_{i=0}^{\s-1} i\frac{\rho^{i}}{i!}
    = \sum_{i=1}^{\s-1} \frac{\rho^{i}}{(i-1)!}
    =\sum_{k=0}^{\s-2} \frac{\rho^{k+1}}{k!}
\end{align}
\end{subequations}
Combining these terms recovers $S$:
\begin{subequations}
\begin{align}
    VII &:= I - III
    = \s \frac{\rho^{\s}}{(\s-1)!}
    \\
    VIII &:= IV - VI 
    = \sum_{k=0}^{\s-2} k \frac{\rho^{k
    +1}}{k!}
    = \sum_{k=1}^{\s-2} \frac{\rho^{k
    +1}}{(k-1)!}
    = \sum_{j=0}^{\s-3} \frac{\rho^{j
    +2}}{j!}
    \\
    IX & := VIII - II
    = - \frac{\rho^{\s}}{(\s -2)!} \\
    X & := IX + VII 
    = \frac{\rho^{\s}}{(\s -1)!} 
    = \s \frac{\rho^{\s}}{\s!} \\
    S &= V + X 
    = \s \sum_{i=0}^{\s} \frac{\rho^{i}}{i!} 
\end{align}
\end{subequations}
\end{proof}


\begin{thebibliography}{10}

\bibitem{adan2002queueing}
{\sc I.~Adan and J.~Resing}, {\em Queueing theory}, Eindhoven University of
  Technology, 180 (2002).

\bibitem{AubCel84}
{\sc J.~Aubin and A.~Cellina}, {\em Differential Inclusions}, Springer-Verlag,
  1984.

\bibitem{bby18}
{\sc C.~Bandi, D.~Bertsimas, and N.~Youssef}, {\em Robust transient analysis of
  multi-server queueing systems and feed-forward networks}, Queueing Syst., 89
  (2018), pp.~351--413.

\bibitem{hauckdata1}
{\sc R.~C. Barnard, K.~Huang, and C.~D. Hauck}, {\em A mathematical model of
  asynchronous data flow in parallel computers}, IMA Journal of Applied
  Mathematics, 85 (2020), pp.~865--891.

\bibitem{bpt12}
{\sc N.~Bellomo, B.~Piccoli, and A.~Tosin}, {\em Modeling dynamics from a
  complex system viewpoint}, Mathematical Models and Methods in Applied
  Sciences, 22 (2012), p.~1230004.

\bibitem{bhat2015}
{\sc U.~N. Bhat}, {\em An Introduction to Queueing Theory: Modeling and
  Analysis in Applications (Statistics for Industry and Technology)},
  Birkhäuser Verlag, 2015.

\bibitem{CheGol59}
{\sc W.~Cheney and A.~A. Goldstein}, {\em Proximity maps for convex sets}, Proc
  Amer Math Soc, 10 (1959), pp.~448--450.

\bibitem{ckpt2014}
{\sc A.~Chertock, A.~Kurganov, A.~Polizzi, and I.~Timofeyev}, {\em Pedestrian
  flow models with slowdown interactions}, Mathematical Models and Methods in
  Applied Sciences, 24 (2014), pp.~249--275.

\bibitem{mcbook13}
{\sc W.-K. Ching, X.~Huang, M.~Ng, and T.~Siu}, {\em Markov chains}, Springer,
  2013.

\bibitem{css00}
{\sc D.~Chowdhury, L.~Santen, and A.~Schadschneider}, {\em Statistical physics
  of vehicular traffic and some related systems}, Physics Reports, 329 (2000),
  pp.~199--329.

\bibitem{daneshvar2021}
{\sc M.~Daneshvar}, {\em Modeling Information Flow Using Multi-Stage Queueing
  Model}, PhD thesis, University of Houston, https://hdl.handle.net/10657/9201,
  2021.

\bibitem{omega2016b}
{\sc M.~Defraeye and I.~{Van Nieuwenhuyse}}, {\em Staffing and scheduling under
  nonstationary demand for service: A literature review}, Omega, 58 (2016),
  pp.~4--25.

\bibitem{lega13}
{\sc A.~Galante and D.~Levy}, {\em Modeling selective local interactions with
  memory}, Physica D: Nonlinear Phenomena, 260 (2013), pp.~176 -- 190.
\newblock Emergent Behaviour in Multi-particle Systems with Non-local
  Interactions.

\bibitem{grimstir92}
{\sc G.~Grimmett and D.~Stirzaker}, {\em Probability and Random Processes, 2nd
  ed.}, Oxford University Press, 1992.

\bibitem{HaiNorWan00}
{\sc E.~Hairer, S.~N{\o}rsett, and G.~Wanner}, {\em Solving Ordinary
  Differential Equations {I}: Nonstiff problems}, Springer, Berlin, 2000.

\bibitem{hst2014}
{\sc C.~Hauck, Y.~Sun, and I.~Timofeyev}, {\em On cellular automata models of
  traffic flow with look-ahead potential}, Stochastics and Dynamics, 14 (2014),
  p.~1350022.

\bibitem{hauckdata2}
{\sc C.~D. Hauck, M.~Herty, and G.~Visconti}, {\em Qualitative properties of
  mathematical model for data flow}, Networks and Heterogeneous Media, 16
  (2021), pp.~513--533.

\bibitem{helbing2001}
{\sc D.~Helbing}, {\em Traffic and related self-driven many-particle systems},
  Rev. Mod. Phys., 73 (2001), pp.~1067--1141.

\bibitem{hogg1977probability}
{\sc R.~V. Hogg, E.~A. Tanis, and D.~L. Zimmerman}, {\em Probability and
  statistical inference}, vol.~993, Macmillan New York, 1977.

\bibitem{kotabj19}
{\sc B.~Karamched, W.~Ott, I.~Timofeyev, R.~N. Alnahhas, M.~R. Bennett, and
  K.~Josic}, {\em Boundary-driven emergent spatiotemporal order in growing
  microbial colonies}, Physica D, 395 (2019), pp.~1--6.

\bibitem{moran1962}
{\sc P.~A.~P. Moran}, {\em The statistical processes of evolutionary theory},
  Oxford, Clarendon Press, 1962.

\bibitem{norris97}
{\sc J.~Norris}, {\em Markov Chains}, Cambridge University Press, 1997.

\bibitem{otw2019}
{\sc W.~Ott, I.~Timofeyev, and T.~Weber}, {\em Stochastic and coarse-grained
  two-dimensional modeling of directional particle movement}, Physica D, 402
  (2019).

\bibitem{piegorsch1990maximum}
{\sc W.~W. Piegorsch}, {\em Maximum likelihood estimation for the negative
  binomial dispersion parameter}, Biometrics,  (1990), pp.~863--867.

\bibitem{RocWet98}
{\sc R.~T. Rockafellar and R.~J. Wets}, {\em Variational Analysis}, vol.~317 of
  Grundlehren der mathematischen Wissenschaften, Springer, 1998.

\bibitem{rothkopforen1977}
{\sc M.~H. Rothkopf and S.~S. Oren}, {\em The nonstationary {M/M/s} queue},
  Report ARG 77-10, Xerox Palo Alto Research Center, Palo Alto, California,
  (1977).

\bibitem{rothkopforen1979}
\leavevmode\vrule height 2pt depth -1.6pt width 23pt, {\em A closure
  approximation for the nonstationary {M/M/s} queue}, Management Science, 25
  (1979), pp.~522--534.

\bibitem{schchni10}
{\sc A.~Schadschneider, D.~Chowdhury, and K.~Nishinari}, {\em Stochastic
  Transport in Complex Systems: From Molecules to Vehicles}, Elsevier, 2010.

\bibitem{omega2016a}
{\sc J.~A. Schwarz, G.~Selinka, and R.~Stolletz}, {\em Performance analysis of
  time-dependent queueing systems: Survey and classification}, Omega, 63
  (2016), pp.~170--189.

\bibitem{soka06}
{\sc A.~Sopasakis and M.~Katsoulakis}, {\em Stochastic modeling and simulation
  of traffic flow: Asymmetric single exclusion process with arrhenius
  look-ahead dynamics}, SIAM J. Appl. Math., 66 (2006), pp.~921--944.

\bibitem{spitzer1970}
{\sc F.~Spitzer}, {\em Interaction of {Markov} processes}, Advances in
  Mathematics, 5 (1970), pp.~246--290.

\bibitem{sce11}
{\sc Y.~Sun, R.~Caflisch, and B.~Engquist}, {\em A multiscale method for
  epitaxial growth}, SIAM Multiscale Model. Simul., 9 (2011), pp.~335--354.

\bibitem{sunti14}
{\sc Y.~Sun and I.~Timofeyev}, {\em Kinetic monte carlo simulations of {1D} and
  {2D} traffic flows: Comparison of two look-ahead potentials}, Phys. Rev. E.,
  89 (2014), p.~052810.

\bibitem{lewe14}
{\sc D.~Weinberg and D.~Levy}, {\em Modeling selective local interactions with
  memory: Motion on a {2D} lattice}, Physica D: Nonlinear Phenomena, 278
  (2014), pp.~13 -- 30.

\end{thebibliography}
\end{document}